\pdfoutput=1
\documentclass[english]{jnsao}

\usepackage[utf8]{inputenc}
\usepackage{enumitem}
\setenumerate{label=(\roman*)}

\numberwithin{equation}{section} 
\usepackage[capitalise,nameinlink]{cleveref}

\usepackage{booktabs}
\usepackage{algorithm}
\usepackage{algorithmicx,algpseudocode}

\usepackage{siunitx}
\newenvironment{msc}{\vspace*{-0.75cm}\small\quotation\noindent{\normalfont\sectfont\nobreak MSC (2020)\quad}}{\endquotation\vskip0.7cm}
\newcommand{\mscLink}[1]{\href{http://www.ams.org/mathscinet/msc/msc2010.html?t=#1}{#1}}
\newenvironment{keywords}{\vspace*{-0.75cm}\small\quotation\noindent{\normalfont\sectfont\nobreak Keywords\quad}}{\endquotation\vskip0.7cm}

\newcommand\be{\begin{equation}}
\newcommand\ee{\end{equation}}

\renewcommand{\subset}{\subseteq}

\newcommand{\dx}{\,\text{\rmfamily{}\upshape{}d}x}

\newcommand{\dt}{\,\text{\rmfamily{}\upshape{}d}t}

\def\N{\mathbb  N}
\def\R{\mathbb  R}

\def\L{\mathcal  L}

\def\argmin{\operatorname*{arg\, min}}

\def\dom{\operatorname{dom}}

\def\dtol{\delta_{\mathrm{tol}}}

\newcommand{\PP}[1]{{\normalfont({\bfseries P\!}${}_{#1}$)}}
\def\PB{\PP{B}} 

\newcolumntype{L}{>{$}l<{$\quad}}
\newcolumntype{R}{>{$}r<{$\quad}}
\newcolumntype{C}{>{$}c<{$}}

\manuscriptlicense{CC-BY 4.0}

\manuscripteprinttype{arxiv}
\manuscripteprint{2211.12246v5}

\manuscriptsubmitted{2023-10-05}
\manuscriptaccepted{2024-06-24}
\manuscriptvolume{5}
\manuscriptnumber{12366}
\manuscriptyear{2024}
\manuscriptdoi{10.46298/jnsao-2024-12366}

\begin{document}

\title{A topological derivative-based algorithm to solve optimal control problems with \texorpdfstring{$\scriptstyle L^0(\Omega)$}{L⁰(Ω)} control cost}
\shorttitle{A topological derivative-based algorithm for \texorpdfstring{$\scriptstyle L^0$}{L⁰} problems}

\author{Daniel Wachsmuth%
\thanks{Institut f\"ur Mathematik,
Universit\"at W\"urzburg,
97074 W\"urzburg, Germany,
		\email{daniel.wachsmuth@uni-wuerzburg.de},
		\orcid{0000-0001-7828-5614}}}

\shortauthor{Daniel Wachsmuth}

\date{2024-02-20}

\acknowledgements{
This research was supported by the German Research Foundation (DFG) under grant number WA 3626/3-2
within the priority program ``Non-smooth and Complementarity-based Distributed Parameter
Systems: Simulation and Hierarchical Optimization'' (SPP 1962).%
}

\maketitle

\begin{abstract}
In this paper, we consider optimization problems with $L^0$-cost of the controls.
Here, we take the support of the control as independent optimization variable.
Topological derivatives of the corresponding value function with respect to variations of the support are derived.
These topological derivatives are used in a
novel gradient descent algorithm with Armijo line-search.
Under suitable assumptions, the algorithm produces a minimizing sequence.
\end{abstract}


\begin{msc}
\mscLink{49M05},   	
\mscLink{49K40},   	
\mscLink{65K10}   	
\end{msc}

\section{Introduction}

In this paper we are interested in the following optimal control problem: Minimize
\begin{equation}\label{eqintro_ex_001}
 \min  \frac12 \|y-y_d\|_{L^2(\Omega)}^2 + \frac\alpha2 \|u\|_{L^2(\Omega)}^2 + \beta \|u\|_0
\end{equation}
over all $(y,u)$ satisfying
\begin{equation}\label{eqintro_ex_002}\begin{aligned}
 -\Delta  y &= u \quad \text{ on }\Omega\\
 y &= 0 \quad \text{ on }\partial\Omega
\end{aligned}\end{equation}
and
\begin{equation}\label{eqintro_ex_003}
  u_a \le u \le u_b.
\end{equation}
Here, $\|u\|_0$ is the measure of the support of $u$. This optimal control problem can be interpreted in the context
of optimal actuator placement: Find a (possibly small) measurable set $A\subset\Omega$ such that controls supported on $A$
can still minimize a certain objective functional.
We remark that the elliptic equation in \eqref{eqintro_ex_002} can be replaced by other types of partial differential equations,
for example parabolic or hyperbolic equations.
A control support optimization subject to the wave equation with terminal contraints is performed in \cite{Munch2008,Munch2009}.

In this work, we will take the support of the control $u$ as own optimization variable $A \subset \Omega$.
In addition, we will allow for a more general control problem as above. The abstract problem we are interested in is:
Minimize with respect to $u\in L^2(\Omega)$ and measurable $A\subset \Omega$ the functional
\begin{equation}\label{eq002}
 J(u,A):=\frac12\|S(\chi_Au) - y_d\|_Y^2 + \int_\Omega g(u(x)) + \chi_A(x)\beta(x) \dx,
\end{equation}
where $S:L^2(\Omega) \to Y$ is a solution operator of a linear partial differential equation, $Y$ is a Hilbert space, $y_d\in Y$ is given,
$g:\R \to \bar \R$ is a strongly convex function, and $\beta\in L^1(\Omega)$.
Then for fixed measurable $A\subset \Omega$ the map $u \mapsto  J(u,A)$ is convex, which is crucial for our analysis.
As already mentioned, we can allow for time-dependent differential equations as well, see \cref{ex_parabolic,ex_hyperbolic} below.

The main contribution of this paper is to derive an algorithm to solve \eqref{eq002}.
In contrast to existing approaches \cite{ItoKunisch2014,KaliseKunischSturm2018, Wachsmuth2019},
the algorithm proposed in this paper produces a minimizing sequence for \eqref{eq002}.
Such an algorithm is not available in the literature.

Given $A\subset \Omega$ measurable, the functional $u\mapsto J(u,A)$ admits minimizers, and we can study the
value function
\begin{equation}\label{eq003}
J(A):= \min_{u\in L^2(\Omega)} J(u,A),
\end{equation}
where the minimization is carried out over measurable sets $A\subset\Omega$.
We will investigate topological derivatives of the value function.
In additions, we are interested in the shape optimization problem
\begin{equation}\label{eq004}
\min_{A\subset \Omega} J(A).
\end{equation}
The topological derivative $DJ(A)$ is the main result of \cref{thm_top_der}. It can be extended to non-strongly convex $g$, see \cref{thm_top_der_nsc}.
These results generalize available results in the literature \cite{Amstutz2011,KaliseKunischSturm2018,Munch2008,Munch2009}, as we allow for non-smooth $g$ and incorporate control constraints.
In comparison to \cite{Amstutz2011,KaliseKunischSturm2018}, we will use less smoothness assumptions, in particular no continuity of controls and adjoints is required.
An optimality condition for \eqref{eq004} can be given in terms of the topological derivatives as follows: If $B$ is a solution of \eqref{eq004}
then $DJ(B)\ge0$, see \cref{thm_opt_con}.

The concept of topological derivatives goes back to the seminal work \cite{SokoOwskiZochowski1999}.
It was applied to an optimal control problem in \cite{SokoOwskiZochowski1999ctrl},
where the topological derivative with respect to changes of the domain but not of the control domain was computed.
In these works, asymptotic analysis with respect to radius of small inclusions/exclusions was performed.
For an introduction and overview of available results regarding topological derivatives, we refer to the monographs \cite{NovotnySokoOwski2013,NovotnySokoOwskiZochowski2019_book} and the recent introductory exposition \cite{Amstutz2021}.

While topological perturbations of source terms is a well understood topic,
see, e.g., \cite[Theorem 2.1]{Amstutz2021},
this is not true for
topological perturbations of the control domain in control problems like \eqref{eq004}.
The topological derivative of a value function of an optimal control problem subject to the wave equation with terminal constraints was
given in \cite{Munch2008,Munch2009} for $A=\Omega$ without proof.
Topological derivatives of value functions of optimal control problems were derived in \cite{KaliseKunischSturm2018}
for problems without control constraints, however the result and its proof are wrong.
In particular, their topological derivative evaluated at $A$ is zero on the complement of $A$
for $L^2$-controls in space,
\cite[Corollary 4.1]{KaliseKunischSturm2018},
which was corrected in the mean-time with a new version on arxiv.
In addition, the underlying abstract theory only allows to compute the topological derivative
at one fixed point, which necessitates continuity assumptions on that point $x\in \Omega$.
In our proof, we get the topological derivative at almost all $x\in \Omega$ at once using the Lebesgue differentiation theorem.
Moreover, we can allow for control constraints and non-smooth functions $g$.


In addition to the development of the topological derivative, we also investigate an algorithm to solve the
problem at hand. In the algorithm, variations of a given set $A_k \subset \Omega$ are performed at points, where
the topological derivative $DJ(A_k)$ has the wrong sign.
Let $\rho_k:=DJ(A_k)^-$ be the residual in the optimality condition for \eqref{eq004}, and set $R_k:=\{x: \ \rho_k\ne0\}$.
Then the new iterate $A_{k+1}$ is defined as $A_{k+1} :=A_k \triangle D_{k,t}$, where $\triangle$ denotes the symmetric difference of sets. The set
$D_{k,t} \subset R_k$ is chosen such that
\[
 \|\rho_k\|_{L^1(D_{k,t})} \ge t  \|\rho_k\|_{L^1(\Omega)} ,
 \qquad  |D_{k,t}| \le  t |R_k|,
\]
where $t\in (0,1]$ is determined by a linesearch strategy to guarantee a sufficient decrease of $J(A_{k+1})$.
These choices enable a satisfying convergence theory: the method produces a minimizing sequence, see \cref{thm_convergence}.
A related algorithm can be found in \cite{CeaGioanMichel1973,CeaGioanMichel1974}. We comment on the differences to our method in  \cref{rem_algorithms}.

Our choice of $D_{k,t}$ is different to the classical topological optimization algorithm,
where one sets
\[
 D_{k,t} := \{ x : \ |\rho_k(x)| \ge t \|\rho_k\|_{L^\infty(\Omega)}\}.
\]
The parameter $t$ is determined to enforce a volume constraint, see, e.g., \cite{CeaGarreauGuillaumeMasmoudi2000,GarreauGuillaumeMasmoudi2001,GuillaumeSidIdris2002},
or to ensure decrease of the functional, e.g., \cite{CarpioRapun2012,HriziHassine2018,HassineJanMasmoudi2007,HintermullerLaurain2010}.
An alternative algorithmic idea is to introduce topological changes near local maxima of $|\rho_k|$, see e.g.,
\cite{Amstutz2005,BenAbdaHassineJaouaMasmoudi2009}.
No convergence results are given in these works.

Another method was introduced in \cite{AmstutzA2006}: there a (simplified) level-set method was suggested, where the evolution of the level-set function
is done using the topological derivative. This method is applied to a wide variety of problems, see e.g., \cite{Amstutz2011,AmstutzA2006,FulmanskiLaurainScheidSokoOwski2007,NovotnySokoOwskiZochowski2019}.
An convergence proof can be found in \cite{Amstutz2011}. However, for the proof the functional has to be replaced by its $H^s$-lower semicontinuous envelope.

Topological optimization problems are related to binary control or 0-1-optimi\-zation problems.
This connection is exploited in \cite{Amstutz2010,AmstutzLaurain2013}.
An trust-region method to solve such problems is analyzed in the recent contributions \cite{HahnLeyfferSager2023,MannsHahnKirchesLeyfferSager2023},
where it is proven that a certain criticality measure converges to zero during the iteration,
which is comparable to our result. 

Let us emphasize that the convergence analysis in this paper is enabled by the particular structure of the
problem. It is expected that the analysis carries over to related problems (e.g., source identification problems).
However, it is unclear how to transfer these results to harder problems,
where the optimization variables appear in the main part of the operator as in, e.g.,
material or topology optimization problems.
In fact, the question of convergence of topological derivative-based methods is mentioned as an open problem in \cite[Section 5]{NovotnySokoOwskiZochowski2019}.

\subsubsection*{Notation}

We will denote the Lebesgue measure of a measurable set $A\subset \R^d$ by $|A|$.
For $r>0$ and $x\in \R^d$, let $B_r(x)$ be the open ball with radius $r$ centered at $x$. Its Lebesgue measure will be denoted by $|B_r|$.
We set $\bar \R:=\R \cup \{+\infty\}$. For a function $g:\R\to \bar \R$, we set $\dom g:=\{u\in \R: \ g(u)<+\infty\}$.
The subdifferential of a convex function $g$ at $u$ will be denoted by $\partial g(u)$.
We will write $x^+ := \max(x,0)$ and $x^-:=\min(x,0)$ for $x\in\R$.

\subsubsection*{Convention}

In the following, we will always take Lebesgue measurable subsets of $\Omega$ only, without explicitly mentioning.

\section{Assumptions and preliminary results}

Throughout this paper, we will work with the following assumptions concerning the problem \eqref{eq002}

\begin{enumerate}[label=\bfseries (A\arabic*)]
 \item\label{ass1_Omega}
	$\Omega \subset \R^d$ is Lebesgue measurable with $|\Omega|<\infty$.
 \item\label{ass1_y}
	$Y$ is a real Hilbert space, $S\in \L(L^2(\Omega) ,Y)$, $y_d \in Y$.
 \item\label{ass1_g}
	$g:\R\to \bar \R$ is proper, convex, lower semi-continuous. In addition, $g(u)\ge0$ for all $u\in \R$, and $g(u)=0$ if and only if $u=0$.
 \item\label{ass1_strongc}
	There is $\mu>0$ such that
 \[
  \frac\mu2 \lambda(1-\lambda) |u-v|^2 + g(\lambda u+(1-\lambda)v) \le \lambda g(u) + (1-\lambda) g(v)
 \]
 for all $u,v\in \dom g$, $\lambda \in (0,1)$.
 \item\label{ass1_smooth}
	There is $q > 6$ such that $S^*S \in \L(L^2(\Omega), L^q(\Omega))$ and $S^*y_d\in L^q(\Omega)$, where $S^*\in \L(Y,L^2(\Omega))$ denotes the Hilbert space-adjoint of $S$.

 \item\label{ass1_beta}
	$\beta\in L^1(\Omega)$.
\end{enumerate}

Let us comment on these assumptions. As we plan to use the Lebesgue differentiation theorem, we assume that the underlying measure space
is induced by the Lebesgue measure of $\R^d$ in \ref{ass1_Omega}.
Conditions \ref{ass1_y}, \ref{ass1_g}, \ref{ass1_strongc} imply the well-posedness of the problem $\min_{u\in L^2(\Omega)} J(u,A)$ for fixed $A$.
Assumption \ref{ass1_strongc} is strong convexity of the function $g$. The results of the paper are still valid in the non-strongly convex case ($\mu=0$)
under slightly strengthened assumptions on $g$ and $S$, we will comment on this in  \cref{sec_nonstronglyconvex}.
Condition \ref{ass1_smooth} is a mild assumption on the smoothing properties of the operators $S$ and $S^*$.
It implies that certain remainder terms in the expansion of topological derivatives are of higher order, see \cref{thm_top_der}.

We will explicitly mention in upcoming, important results (theorems and propositions), which of these assumptions are used.
If the strong convexity assumption is not mentioned then $\mu$ can be taken equal to zero.

\begin{example}\label{ex_elliptic}
 Let us comment on the fulfillment of these assumptions for the introductory example \eqref{eqintro_ex_001}--\eqref{eqintro_ex_003}.
 Let $\Omega \subset \R^d$ be a bounded domain.
 The partial differential equation \eqref{eqintro_ex_002} is uniquely solvable in the weak sense, so that the mapping $S:u \mapsto y$
 is linear and continuous from $L^2(\Omega)$ to $H^1_0(\Omega)$.
 The functional \eqref{eqintro_ex_001} requires the choice $Y:=L^2(\Omega)$, so that $S$ will be considered as operator on $L^2(\Omega)$, which makes $S$ self-adjoint.
 Due to the classical result \cite{Stampacchia1965}, $S$ is continuous from $L^2(\Omega)$ to $L^\infty(\Omega)$ if $d\le 3$.
 By \cite[Theorem 18]{BrezisStrauss1973}, $S$ and $S^*$ are in $\L( L^2(\Omega), L^{10/3}(\Omega))$
 and $\L(L^{10/3}(\Omega)),L^{10}(\Omega))$ for all $d\le 10$.
 And \ref{ass1_y} and \ref{ass1_smooth} are satisfied for this example provided $y_d \in L^{10/3}(\Omega)$ and $d\le 10$.
\end{example}

\begin{example}\label{ex_parabolic}
 The following distributed control problem subject to a parabolic equation can also be put into the framework above: Minimize $\int_Q \frac12(y-y_d)^2 + g(u) \dx\dt$
 subject to the parabolic equation $y_t - \Delta y =u$ in $Q$, $y=0$ on $(0,T)\times \partial D$, and $y(\cdot, 0)=0$ on $D$, where $D\subset \R^d$ is a bounded domain, $T>0$, and $Q:=(0,T) \times D$.
 We set $Y:=L^2(Q)$. The corresponding solution operator $S$ is continuous from $L^2(Q)$ to $L^2(0,T; H^1_0(D))\cap H^1(0,T; H^{-1}(D)) $.
 Its adjoint operator $S^*$ is given as the solution operator of the adjoint equation, i.e., $p=S^*z$, where $p$ solves $-p_t - \Delta p =z$, $p(T)=0$.
 According to \cite[Theorem 2.3]{CasasWachsmuth2022}, $S$ and $S^*$ are in $\L( L^2(Q), L^{10/3}(Q))$
 and $\L(L^{10/3}(Q)),L^{10}(Q))$ for all $d\le 8$.
 If $y_d \in L^{10/3}(Q)$ then \ref{ass1_y} and \ref{ass1_smooth} are fulfilled for all $d\le 8$.
 Note that in this example the control domain $\Omega$ has to be set to $\Omega:=Q = (0,T) \times D$.
\end{example}

\begin{example}\label{ex_hyperbolic}
One can also consider distributed control of the wave equation,
where we are interested in minimizing the same functional as in \cref{ex_parabolic}
subject to the wave equation
$y_{tt} - \Delta y =u$ in $Q$, $y=0$ on $(0,T)\times \partial D$, and $y(\cdot, 0)=y_t(\cdot,0)=0$ on $D$,
where $D\subset \R^d$ is a bounded domain, $T>0$, and $Q:=(0,T) \times D$.
Given $u\in L^2(Q)$, there is a unique weak solution $y\in L^\infty(0,T;H^1_0(D))$,
where $L^\infty(0,T;H^1_0(D))$ is continuously embedded into $L^q(Q)$ for all $q<\infty$ if $d\le 2$.
And \ref{ass1_y} and \ref{ass1_smooth} are fulfilled for $d\le2$.
Using improved regularity results and Strichartz estimates, see, e.g., \cite[Section 7.2]{Evans2010}, it should be possible to relax the requirement on $d$.
\end{example}

\subsection{Existence of minimizers of $J$ for fixed $A$}

Let $A\subset \Omega$ be given.
Here, we consider the problem
\begin{equation} \label{eq_p_a}
\tag{\mbox{\bfseries P\!${}_A$}}
 \min_{u\in L^2(\Omega)} J(u,A).
\end{equation}
where $J$ is given by \eqref{eq002}.
Note that due to the construction of $J$ and \ref{ass1_g}, we have
\begin{equation}\label{eq_J_chi}
 J(\chi_A u, A) \le J(u,A)
\end{equation}
for all $u\in L^2(\Omega)$.
Due to strong convexity of $g$ and $g(0)=0$ by \ref{ass1_g}, \ref{ass1_strongc}, we have
\begin{equation}\label{eq_g_coercive}
 g(u) \ge \frac\mu2 |u|^2 \quad\forall u\in \dom g.
\end{equation}
\begin{proposition}\label{prop_existence}
Assume \ref{ass1_Omega}, \ref{ass1_y}, \ref{ass1_g},  \ref{ass1_strongc}.
 Let $A\subset \Omega$ be given.  Then there is a uniquely determined minimizer $u_A$ of \eqref{eq_p_a}.
 Moreover,
\begin{equation}\label{eq_optcon_utimeschi}
 \chi_A u_A = u_A.
\end{equation}
\end{proposition}
\begin{proof}
 Due to \eqref{eq_g_coercive}, minimizing sequences of $J(\cdot,A)$ are bounded in $L^2(\Omega)$. In addition,
 $u\mapsto J(u,A)$ is weakly lower semi-continuous from $L^2(\Omega)$ to $\R$ because of \ref{ass1_y} and \ref{ass1_g}.
 The existence of solutions follows now by standard arguments. Uniqueness of solutions is  a consequence of strong convexity of g \ref{ass1_strongc}.
 The last claim follows from \eqref{eq_J_chi}.
\end{proof}


In all what follows, we will not make use of the unique solvability of \eqref{eq_p_a}.
We will just use that $u_A$ is any solution of \eqref{eq_p_a}.

\subsection{Optimality conditions for \eqref{eq_p_a}}

Let $A\subset \Omega$ be given, and let $u_A$ be a solution of \eqref{eq_p_a}.
Let us denote the associated state by
\begin{equation}\label{eq_def_y}
 y_A := S( \chi_A u)
\end{equation}
and adjoint state by
\begin{equation}\label{eq_def_p}
 p_A := S^*(y_A - y_d)  = S^*(S(\chi_A u)-y_d).
\end{equation}
Let $u\in L^2(\Omega)$ and $B\subset \Omega$ be given. Let $y:=S( \chi_B u)$.
Then by elementary calculations, we
find
\begin{equation}\label{eq_exp_ymyd}
\begin{aligned}
 \frac12\|y - y_d\|_Y^2 - \frac12\|y_A - y_d\|_Y^2 &= (y_A-y_d, y-y_A)_Y + \frac12\|y-y_A\|_Y^2 \\
 & = (p_A, \chi_B u- \chi_Au_A)+\frac12\|y-y_A\|_Y^2.
\end{aligned}
\end{equation}
For $B=A$, we get
\[
 \frac12\|y - y_d\|_Y^2 - \frac12\|y_A - y_d\|_Y^2  = (p_A, \chi_A( u- u_A) )+\frac12\|y-y_A\|_Y^2.
\]
Hence,  $\chi_Ap_A \in L^2(\Omega)$ is the Fr\'echet derivative of $u \mapsto  \frac12\|S( \chi_A u) - y_d\|_Y^2$ at $u_A$.

\begin{proposition}
Assume \ref{ass1_Omega}, \ref{ass1_y}, \ref{ass1_g}.
 Let $A\subset \Omega$ and let $u_A$ be a solution of \eqref{eq_p_a}.
 Let $p_A$ be given by \eqref{eq_def_p}.
 Then it holds
 \begin{equation}\label{eq_optcon_pw}
  - \chi_A(x)p_A(x) \in \partial g(u(x)) \quad \text{ for almost all } x\in \Omega
 \end{equation}
and
 \begin{equation}\label{eq_optcon_pmp}
 u_A(x) = \argmin_{u\in \R} \chi_A(x)p_A(x)\cdot u + g(u) \quad \text{ for almost all } x\in \Omega.
\end{equation}
\end{proposition}
\begin{proof}
Let us denote $G(u):= \int_\Omega g(u(x))\dx$.
As argued above,  $\chi_Ap_A \in L^2(\Omega)$ is the Fr\'echet derivative of $u \mapsto  \frac12\|S( \chi_A u) - y_d\|_Y^2$ at $u_A$.
Then by well-known results, see, e.g., \cite[Proposition II.2.2]{EkelandTemam1999}, we get
$-\chi_A p_A \in \partial G(u_A)$.
Using \cite[Corollary 3E]{Rockafellar1976}, this is equivalent to the pointwise a.e.\@ inclusion \eqref{eq_optcon_pw}, which in turn is equivalent
to \eqref{eq_optcon_pmp}.
\end{proof}

Condition \eqref{eq_optcon_pmp} can be interpreted as Pontryagin's maximum principle for \eqref{eq_p_a}.

\subsection{Boundedness results for solutions of \eqref{eq_p_a}}

In this section, we will derive bounds on $(u_A,y_A,p_A)$ that are uniform with respect to $A\subset \Omega$.

\begin{lemma}\label{lem_bound_uy}
 There is  $M>0$ such that
 \[
  \|y_A - y_d\|_Y  + \|u_A\|_{L^2(\Omega)} \le M
 \]
for all $A\subset \Omega$.
\end{lemma}
\begin{proof}
 This follows directly from $J(A,u_A) \le J(A,0)$ and \eqref{eq_g_coercive}.
\end{proof}


\begin{corollary}\label{cor_linfty_bound}
 There is $P>0$ such that
 \[
  \|u_A\|_{L^q(\Omega)}\le P , \quad \|p_A\|_{L^q(\Omega)} \le P
 \]
for all $A\subset \Omega$, where $q$ is from \ref{ass1_smooth}.
\end{corollary}
\begin{proof}
First, we have
\[\begin{aligned}
\|p_A\|_{L^q(\Omega)} &\le \|S^*S\|_{\L(L^2(\Omega), L^q(\Omega))} \|u_A\|_{L^2(\Omega)} + \|S^*y_d\|_{L^q(\Omega)} \\
&\le  \|S^*S\|_{\L(L^2(\Omega), L^q(\Omega))}M+ \|S^*y_d\|_{L^q(\Omega)}
\end{aligned}\]
by \ref{ass1_smooth}
with $M$ as in \cref{lem_bound_uy}.
Using \eqref{eq_optcon_pmp} with $u=0$, \ref{ass1_g}, and \eqref{eq_g_coercive}, we have for almost all $x\in \Omega$
\[
\frac\mu2 |u_A(x)|^2 \le g(u_A(x)) \le - \chi_A(x) p_A(x) u_A(x)
\]
which implies $\frac\mu2 |u_A(x)| \le |p_A(x)|$ and  $\|u_A\|_{L^q(\Omega)}\le 2\mu^{-1}\|p_A\|_{L^q(\Omega)} $.
\end{proof}

\section{Analysis of the value function}
\label{sec3}

In this section, we will investigate stability properties of $A\mapsto (u_A,y_A,p_A)$, where $y_A$ and $p_A$ solve \eqref{eq_def_y} and \eqref{eq_def_p}.
The goal is to derive formulas for the topological derivative of $A\mapsto J(A)$, where $J(A)$ is the value function
defined in \eqref{eq003} by
\[
 J(A) = \min_{u\in L^2(\Omega)} J(u,A).
 \]
For brevity, we refer to tuples $(u_A,y_A,p_A)$, where $u_A$ solves \eqref{eq_p_a} and $y_A,p_A$ are given by \eqref{eq_def_y} and \eqref{eq_def_p}
as solutions of \eqref{eq_p_a}.

\subsection{Sensitivity analysis of \eqref{eq_p_a} with respect to $A$}

%
Let us start with the following preliminary expansion.

\begin{lemma}\label{lem_diff_J_final}
Let $A,B\subset \Omega$, and let $(u_A,y_A,p_A)$ and $(u_B,y_B,p_B)$ be  solutions of \eqref{eq_p_a}
and \PB.
Then it holds
\begin{multline*}
J(A,u_A) - J(B,u_B)  +  \frac12\|y_B-y_A\|_Y^2 \\
=\int_\Omega  g(u_A)-g(u_B) + \chi_Ap_A( u_A - u_B) + (\chi_A-\chi_B)(\beta+ p_Au_B )\dx.
\end{multline*}
\end{lemma}
\begin{proof}
 Doing the expansion of $y\mapsto \frac12\|y-y_d\|_Y^2$ similarly as in \eqref{eq_exp_ymyd}, we have
\begin{equation} \label{eq3101}
 J(A,u_A) - J(B,u_B) +  \frac12\|y_B-y_A\|_Y^2
 = \int_\Omega  g(u_A)-g(u_B) + p_A( \chi_Au_A - \chi_Bu_B) + (\chi_A-\chi_B) \beta\dx.
\end{equation}
In addition, we have
\[
 \int_\Omega p_A( \chi_Au_A - \chi_Bu_B)\dx  =\int_\Omega \chi_A p_A( u_A - u_B) - (\chi_B-\chi_A)p_Au_B \dx ,
\]
which is the claim.
\end{proof}

%

\begin{lemma}\label{lem_diff_uy}
Let $A,B\subset \Omega$, and let $(u_A,y_A,p_A)$ and $(u_B,y_B,p_B)$ be solutions of \eqref{eq_p_a}
and \PB.
Then it holds
\[
\mu\|u_B-u_A\|_{L^2(\Omega)}^2 +  \|y_B-y_A\|_Y^2
\le\int_\Omega  (\chi_A-\chi_B)(p_Au_B - p_Bu_A)\dx
\]
with $\mu\ge0$ as in \ref{ass1_strongc}.
\end{lemma}
\begin{proof}
Due to \cref{lem_diff_J_final}, we have
\begin{multline*}
J(A,u_A) - J(B,u_B)  +  \frac12\|y_B-y_A\|_Y^2 \\
=\int_\Omega  g(u_A)-g(u_B) + \chi_Ap_A( u_A - u_B) + (\chi_A-\chi_B)(\beta+ p_Au_B )\dx
\end{multline*}
as well as
\begin{multline*}
J(B,u_B) - J(A,u_A)  +  \frac12\|y_A-y_B\|_Y^2 \\
=\int_\Omega  g(u_B)-g(u_A) + \chi_Bp_B( u_B - u_A) + (\chi_B-\chi_A)(\beta+ p_Bu_A )\dx.
\end{multline*}
Adding both equations gives
\[
 \|y_A-y_B\|_Y^2 =\int_\Omega (\chi_Ap_A-\chi_Bp_B)( u_A - u_B) +  (\chi_A-\chi_B)( p_Au_B - p_Bu_A )\dx.
\]
Due the inequality in \ref{ass1_strongc} and the optimality condition \eqref{eq_optcon_pw}, we have
\begin{equation}
 \int_\Omega (\chi_Ap_A-\chi_Bp_B)( u_A - u_B) \dx \le -\mu \|u_A-u_B\|_{L^2(\Omega)}^2,
\end{equation}
and the claim is proven.
\end{proof}

Note that the previous result remains true with $\mu=0$ in the non-strongly convex case.
Now we can prove the main result of this section,
which is a stability estimate of solutions of \eqref{eq_p_a} with respect to variations of $A$ (or $\chi_A$).
In the proof, we will use the fact that for characteristic functions
\[
 \|\chi_A - \chi_B \|_{L^s(\Omega)} =  \|\chi_A - \chi_B \|_{L^1(\Omega)}^{\frac1s} \quad \forall s\in (1,\infty).
\]

\begin{theorem}\label{thm_stability}
Assume \ref{ass1_Omega}, \ref{ass1_y}, \ref{ass1_g}, \ref{ass1_strongc}, \ref{ass1_smooth}.
Then there is a constant $K>0$ such that for
all $A,B\subset \Omega$
\[
  \|p_A-p_B\|_{L^q(\Omega)} +   \|u_A-u_B\|_{L^2(\Omega)} + \|y_B-y_A\|_Y\le K \|\chi_A-\chi_B\|_{L^1(\Omega)}^{\frac12-\frac1q},
\]
where $(u_A,y_A,p_A)$ and $(u_B,y_B,p_B)$ are solutions of \eqref{eq_p_a} and \PB, and $q$ is from  \ref{ass1_smooth}.
\end{theorem}
\begin{proof}
From \ref{ass1_smooth}, we find
\[
\|p_A-p_B\|_{L^q(\Omega)} \le \|S^*S\|_{\L(L^2(\Omega), L^q(\Omega))} \|u_A-u_B\|_{L^2(\Omega)} .
\]
Define $\mu':=\mu / \|S^*S\|_{\L(L^2(\Omega), L^q(\Omega))}^2$.
Let $s$ be such that $\frac1s+\frac1q+\frac12=1$.
From the inequality of \cref{lem_diff_uy}, we obtain with H\"older's inequality
\begin{multline*}
\frac{\mu'}2 \|p_A-p_B\|_{L^q(\Omega)}^2 + \frac \mu2\|u_B-u_A\|_{L^2(\Omega)}^2 + \|y_B-y_A\|_Y^2
\\
\begin{aligned}
& \le \mu\|u_B-u_A\|_{L^2(\Omega)}^2 + \|y_B-y_A\|_Y^2 \\
& \le \int_\Omega  (\chi_A-\chi_B)(p_Au_B - p_Bu_A)\dx\\
&\le \|\chi_A-\chi_B\|_{L^s(\Omega)} ( \|p_A\|_{L^q(\Omega)} \|u_B-u_A\|_{L^2(\Omega)} + \|p_A-p_B\|_{L^q(\Omega)} \|u_A\|_{L^2(\Omega)}) \\
& \le (P+M) \|\chi_A-\chi_B\|_{L^1(\Omega)}^{\frac1s} (\|u_B-u_A\|_{L^2(\Omega)} + \|p_A-p_B\|_{L^q(\Omega)}),
\end{aligned}
\end{multline*}
where $P$ and $M$ are from \cref{cor_linfty_bound} and \cref{lem_bound_uy},
and the claim is proven.
\end{proof}

\subsection{Expansions of the value function}

Let us define $H:\R\times \R \to \bar \R$ by
\[
 H(u,p) := p\cdot u + g(u).
\]
This function reminds of the Hamiltonian of optimal control problems. In the sequel, we need its infimum with respect to $u$,
\[
 \min_{u\in \R} H(u,p) = \min_{u\in \R} \left( p\cdot u + g(u) \right)= - \sup_{u\in \R} \left( -p\cdot u - g(u) \right)
 = -g^*(-p),
\]
where $g^*$ is the convex conjugate to $g$.
The existence of this minimum follows from the properties of $g$ in \ref{ass1_g} and the coercivity estimate \eqref{eq_g_coercive}.
Let us denote this function by $\bar H$, i.e.,
\[
 \bar H(p):=  \min_{u\in \R} H(u,p) = -g^*(-p).
\]
If $(u_A,y_A,p_A)$ is a solution of \eqref{eq_p_a} then we have
$\bar H(p_A(x)) = H(u_A(x),p_A(x))$ for almost all $x\in A$ by \eqref{eq_optcon_pmp}.
We will need some Lipschitz estimates of $\bar H$.

\begin{lemma}\label{lem_barH_lipschitz}
Let $A,B\subset \Omega$, and let $(u_A,y_A,p_A)$ and $(u_B,y_B,p_B)$ be  solutions of \eqref{eq_p_a}
and \PB. Then we have
\[
 \| \bar H(p_A) -  \bar H(p_B)\|_{L^{q/2}(\Omega)} \le P \|p_A-p_B\|_{L^q(\Omega)} \le  PK \|\chi_A-\chi_B\|_{L^1(\Omega)}^{\frac12-\frac1q},
\]
where $P$ and $K$ are from \cref{cor_linfty_bound} and \cref{thm_stability}, respectively.
\end{lemma}
\begin{proof}
Let $p_1,p_2\in \mathbb R$ be given. Let $u_i = \argmin_{v\in [u_a,u_b]}  H(p_i,v)$ for $i=1,2$.
Then we get by the properties of $\bar H$
\[
 \bar H(p_1) \le H(p_1, u_2) = (p_1-p_2)u_2 + H(p_2,u_2) = (p_1-p_2)u_2 + \bar H(p_2).
\]
This implies
 \[
  \bar H(p_2) \le -(p_1-p_2)u_1 + \bar H(p_1)
 \]
by exchanging $(p_1,u_1$) and $(p_2,u_2)$ in the above estimate.
Summarizing, we obtain
\[
 |\bar H(p_1) - \bar H(p_2)| \le |p_1-p_2| \max (|u_1|, |u_2|).
\]
Using \cref{cor_linfty_bound} yields the claim.
\end{proof}

We will proceed with the following expansion of the value function. Note that in the non-strongly convex case
the claim is valid with $\mu=0$.

\begin{lemma}\label{lem_expJ}
Let $A,B\subset \Omega$, and let $(u_A,y_A,p_A)$ and $(u_B,y_B,p_B)$ be solutions of \eqref{eq_p_a}
and \PB.
Then it holds
\[
J(A,u_A) - J(B,u_B) + \frac\mu2 \|u_B-u_A\|_{L^2(A\cap B)}^2 + \frac12\|y_B-y_A\|_Y^2
\le \int_\Omega (\chi_A - \chi_B) (\beta + \bar H(p_A) ) \dx,
\]
where $\mu\ge0$ as in \ref{ass1_strongc}.
 \end{lemma}
\begin{proof}
From \cref{lem_diff_J_final} we get
\begin{multline}\label{eq3601}
J(A,u_A) - J(B,u_B)  +  \frac12\|y_B-y_A\|_Y^2 \\
=\int_\Omega  g(u_A)-g(u_B) + \chi_Ap_A( u_A - u_B) + (\chi_A-\chi_B)(\beta+ p_Au_B )\dx.
\end{multline}
We will now split the integral on the right-hand side into integrals on $A\cap B$, $A\setminus B$, and $B\setminus A$.
This is sufficient as the integrand vanishes outside of $A\cup B$.
For the integral on $A\cap B$, we can use the optimality condition \eqref{eq_optcon_pw} and the (possibly strong) convexity of $g$ to obtain
\[
 \int_{A\cap B}  g(u_A)-g(u_B) + \chi_Ap_A( u_A - u_B) \dx \le -\frac\mu2 \|u_B-u_A\|_{L^2(A\cap B)}^2.
\]
Moreover, $u_B$ vanishes on $A\setminus B$, while $u_A$ vanishes on $B\setminus A$. This allows to simplify
\begin{multline}\label{eq3603}
 \int_\Omega  g(u_A)-g(u_B) + \chi_Ap_A( u_A - u_B) + (\chi_A-\chi_B) p_Au_B \dx
 \\
 \le  -\frac\mu2 \|u_B-u_A\|_{L^2(A\cap B)}^2
 + \int_{A\setminus B}  g(u_A) + p_A u_A \dx
-\int_{B\setminus A}  g(u_B) + p_Au_B\dx
\end{multline}
Using $H$ and $\bar H$, we can write
\begin{multline}\label{eq3604}
 \int_{A\setminus B}  g(u_A) + p_A u_A \dx
-\int_{B\setminus A}  g(u_B) + p_Au_B\dx
\\
=
 \int_{A\setminus B}  \bar H(p_A) \dx
-\int_{B\setminus A}  H(u_B,p_A)\dx
\le
 \int_{A\setminus B}  \bar H(p_A) \dx
-\int_{B\setminus A}  \bar H(p_A)\dx.
\end{multline}
Applying    \eqref{eq3603}  and \eqref{eq3604}, in \eqref{eq3601},
 results in the upper bound
\[
J(A,u_A) - J(B,u_B) + \frac\mu2 \|u_B-u_A\|_{L^2(A\cap B)}^2 + \frac12\|y_B-y_A\|_Y^2
\le
\int_\Omega (\chi_A - \chi_B) (\beta + \bar H(p_A) ) \dx,
\]
which is the claim.
\end{proof}

The next result
gives an expansion of the value function $J(A)$ together with an remainder term that is of
higher order in $\|\chi_A-\chi_B\|_{L^1(\Omega)}$.

\begin{theorem}\label{thm_exp_valuef}
Assume \ref{ass1_Omega}, \ref{ass1_y}, \ref{ass1_g}, \ref{ass1_strongc}, \ref{ass1_smooth}, \ref{ass1_beta}.
Let $A,B\subset \Omega$, and let $(u_A,y_A,p_A)$ and $(u_B,y_B,p_B)$ be solutions of \eqref{eq_p_a}
and \PB.
Then it holds
\begin{multline*}
\left|J(A,u_A) - J(B,u_B) - \int_\Omega (\chi_A - \chi_B) (\beta + \bar H(p_B)) \dx \right| \\
+ \frac\mu2 \|u_B-u_A\|_{L^2(A\cap B)}^2 + \frac12\|y_B-y_A\|_Y^2
\le PK\|\chi_A-\chi_B\|_{L^1(\Omega)}^{3(\frac12-\frac1q)},
\end{multline*}
where $P$, $K$, $q$ are from \cref{cor_linfty_bound}, \cref{thm_stability}, and \ref{ass1_smooth}, respectively.
 \end{theorem}
\begin{proof}
Using the result of \cref{lem_expJ}, we get
\begin{multline*}
J(A,u_A) - J(B,u_B) + \frac\mu2 \|u_B-u_A\|_{L^2(A\cap B)}^2 + \frac12\|y_B-y_A\|_Y^2 \\
\le \int_\Omega (\chi_A - \chi_B) (\beta + \bar H(p_B) - \bar H(p_B) + \bar H(p_A) ) \dx.
\end{multline*}
Using \cref{lem_barH_lipschitz} and \cref{thm_stability}, we can estimate the integral involving $\bar H(p_A)- \bar H(p_B)$ as
\begin{multline}\label{eq3701}
 \int_\Omega (\chi_A - \chi_B) (\bar H(p_A)- \bar H(p_B)) \dx \\
 \le \|\chi_A-\chi_B\|_{L^1(\Omega)}^{1-\frac2q}  \cdot P \|p_A-p_B\|_{L^q(\Omega)} \le  PK \|\chi_A-\chi_B\|_{L^1(\Omega)}^{3( \frac12-\frac1q) }.
 \end{multline}
This results in the upper bound
\begin{multline*}
J(A,u_A) - J(B,u_B) + \frac\mu2 \|u_B-u_A\|_{L^2(A\cap B)}^2 + \frac12\|y_B-y_A\|_Y^2 \\
\le \int_\Omega (\chi_A - \chi_B) (\beta + \bar H(p_B)) \dx + PK\|\chi_A-\chi_B\|_{L^1(\Omega)}^{3(\frac12-\frac1q)}.
\end{multline*}
To obtain a lower bound, we use the result of  \cref{lem_expJ} but with the roles of $A$ and $B$ reversed (and multiplying the resulting inequality by $-1$), which yields
\[
J(A,u_A) - J(B,u_B) - \frac\mu2 \|u_B-u_A\|_{L^2(A\cap B)}^2 - \frac12\|y_B-y_A\|_Y^2 \\
\ge  \int_\Omega (\chi_A - \chi_B) (\beta + \bar H(p_B) ) \dx.
\]
Both inequalities together prove the claim.
\end{proof}

As a by-product of the previous proof,
we get the strengthened stability estimate
\[
  \|u_B-u_A\|_{L^2(A\cap B)} + \|y_B-y_A\|_Y
\le K' \|\chi_A-\chi_B\|_{L^1(\Omega)}^{\frac32(\frac12-\frac1q)},
\]
which improves the exponent from \cref{thm_stability} by a factor $\frac32$.

\begin{remark}
 If $S^* \in \L(Y,L^q(\Omega))$ then the estimate can improved to
\[
  \|u_B-u_A\|_{L^2(A\cap B)} + \|y_B-y_A\|_Y
\le K' \|\chi_A-\chi_B\|_{L^1(\Omega)}^{2(\frac12-\frac1q)}
\]
by estimating $\|p_A-p_B\|_{L^q(\Omega)}$ against $\|y_A-y_B\|_Y$ in the estimate \eqref{eq3701}.
\end{remark}

\section{Topological derivatives}

\begin{definition}
Let $B \subset \Omega$. Then the topological derivative of $J$ at $B$ at the point $x$ is defined by
\[
   DJ(B)(x) = \begin{cases}
				\displaystyle
               \lim_{r\searrow0} \frac{  J(B \cup B_r(x)) - J(B)}{|B_r|} & \text{ if } x\not\in B\\
               \\
				\displaystyle
               \lim_{r\searrow0} \frac{  J(B \setminus B_r(x)) - J(B)}{|B_r|} & \text{ if } x \in B.
              \end{cases}
\]
\end{definition}

The existence of the topological derivative is now a consequence of the expansion in \cref{thm_exp_valuef} and the Lebesgue differentiation theorem.

\begin{theorem}\label{thm_top_der}
Assume \ref{ass1_Omega}, \ref{ass1_y}, \ref{ass1_g}, \ref{ass1_strongc}, \ref{ass1_smooth}, \ref{ass1_beta}.
Let $B\subset \Omega$, and let $(u_B,y_B,p_B)$ be a solution of \PB.

Then for almost all $x \in \Omega$ the topological derivative $ DJ(B)(x)$ exists, and is given by
 \[
  DJ(B)(x) = \sigma(B,x)(\beta(x) + \bar H(p_B(x)))
 \]
with
\[
 \sigma(B,x) := 
 \begin{cases}
	+1 & \text{ if } x\not\in B\\
	-1 & \text{ if } x    \in B.
\end{cases}
\]
\end{theorem}
\begin{proof}
 Let $x_0 \in B$. Let $r>0$. Define $A(x_0,r):=B \setminus B_r(x_0)$.
 Then it follows $\chi_{A(x_0,r)} - \chi_B = - \chi_{B \cap B_r(x_0)}$, which implies $\|\chi_{A(x_0,r)} - \chi_B\|_{L^1(\Omega)} \le |B_r|$.
 Using this in the result of \cref{thm_exp_valuef}, we find
 \begin{equation}\label{eq4201}
  \left|J(A(x_0,r))-J(B) +\int_{ B \cap B_r(x_0)} \beta + \bar H (p_B) \dx \right| \le  PK |B_r|^{3(\frac12-\frac1q)}.
\end{equation}
Let us now define
\[
 v(x_0,r) := \frac1{|B_r|} \int_{B_r(x_0)} \chi_B \cdot (\beta + \bar H (p_B)) \dx.
\]
By the Lebesgue differentiation theorem, we have
\[
 \lim_{r\searrow0} v(x,r) =  \chi_B(x) \cdot (\beta(x) + \bar H (p_B(x)))
\]
for almost all $x\in \Omega$. This implies together with \eqref{eq4201}
\[
 \lim_{r\searrow0} \frac{J(A(x,r))-J(B)}{|B_r|} =  -(\beta(x) + \bar H (p_B(x)))
\]
for almost all $x\in B$. Here we used that $3(\frac12-\frac1q) >1$ by \ref{ass1_smooth}.
This proves the claim for $x \in B$.

The claim for $x\not\in B$ can be proven completely analogously: this time we set $A(x_0,r):=B \cup B_r(x_0)$ for $x_0\not\in B$, which implies
$\chi_{A(x_0,r)} - \chi_B = \chi_{B_r(x_0) \setminus B}$, resulting in the different sign of the topological derivative.
\end{proof}

Note that in contrast to other works, we do not need  to impose continuity of $u_B$ near $x_0$ as in \cite[Corollary 4.1]{KaliseKunischSturm2018},
nor do we need to argue by H\"older continuity of the adjoint as in \cite[Corollary 3.2]{Amstutz2011}.

We can now formulate a necessary optimality condition for \eqref{eq004} using the topological derivative.

\begin{theorem}\label{thm_opt_con}
Assume \ref{ass1_Omega}, \ref{ass1_y}, \ref{ass1_g}, \ref{ass1_strongc}, \ref{ass1_smooth}, \ref{ass1_beta}.
 Let $B$ be a solution of \eqref{eq004}. Then
 \[
  DJ(B)(x) \ge 0  \quad \text{ for a.a. } x\in \Omega.
  \]
\end{theorem}
\begin{proof}
 The result follows immediately from \cref{thm_top_der}.
\end{proof}

\begin{remark}
 Using the celebrated Ekeland's variational principle \cite{Ekeland1974}, the following result can be proven for $\epsilon$-solutions: There is an $\epsilon$-solution,
 such that optimality conditions are satisfied up to $\epsilon$. We briefly sketch the proof.

 Let $V$ be the metric space of characteristic functions $\chi_B$, $B\subset \Omega$ measurable, supplied with the $L^1(\Omega)$-metric, which makes it a complete space.
 Applying \cite[Theorem 1.1]{Ekeland1974} with $\epsilon>0$ and $\lambda=1$
 there is $B_\epsilon \subset \Omega$ such that
 \begin{equation}\label{eq4401}
  J(B_\epsilon) \le \inf_{B\subset \Omega} J(B) + \epsilon
 \end{equation}
and
 \begin{equation}\label{eq4402}
 J(A) \ge J(B_\epsilon) - \epsilon\|\chi_A-\chi_{B_\epsilon}\|_{L^1(\Omega)}
\end{equation}
for all $A\subset \Omega$. Owing to \eqref{eq4401} the set $B_\epsilon$ is then an $\epsilon$-solution of \eqref{eq004}.
Due to  inequality \eqref{eq4402}, we can consider variations of $J(B_\epsilon)$ to obtain
estimates of the topological derivative:

For $x_0 \in \Omega$ and $r>0$, define $A(x_0,r)$ as in the proof of \cref{thm_top_der}.
 Then $\frac1{|B_r|} ( J(A(x_0,r)) - J(B_\epsilon)) \ge -\epsilon$ by \eqref{eq4402},
 which results in $DJ(B_\epsilon)(x_0) \ge -\epsilon$ for almost all $x_0$.
This proves the existence of an $\epsilon$-solution that satisfies the optimality condition up to an $\epsilon$.
\end{remark}

In addition, the defect in the optimality condition of \cref{thm_opt_con} can be used to get an error estimate as follows.

\begin{corollary}\label{lem_error_est}
Assume \ref{ass1_Omega}, \ref{ass1_y}, \ref{ass1_g}, \ref{ass1_beta}.
Let $A \subset \Omega$,  let $(u_A,y_A,p_A)$ be a solution of \eqref{eq_p_a}.
Let the defect $\delta_A$ be defined by
\[
 \delta_A:=\int_{A} \left(\beta + \bar H(p_A)\right)^+ \dx  - \int_{\Omega\setminus A} \left(\beta + \bar H(p_A)\right)^-  \dx
 = -\int_\Omega (DJ(B))^- \dx.
\]
Then we have
\[
 J(A) - \inf_{B\subset \Omega} J(B) \le \delta_A,
\]
and $A$ is a $\delta_A$-solution.
If $B$ is a solution of \eqref{eq004} then we have the error estimate
\[
J(A,u_A) - J(B,u_B)  + \frac12\|y_B-y_A\|_Y^2 +\frac\mu2 \|u_B-u_A\|_{L^2(A\cap B)}^2 \le \delta_A.
\]
\end{corollary}
\begin{proof}
Let $B\subset \Omega$ and $(u_B,y_B,p_B)$ be a solution of \PB.
By \cref{lem_expJ}, we have
\begin{multline*}
J(A,u_A) - J(B,u_B) + \frac\mu2 \|u_B-u_A\|_{L^2(A\cap B)}^2 + \frac12\|y_B-y_A\|_Y^2 \\
\begin{aligned}
& \le \int_\Omega (\chi_A - \chi_B) (\beta + \bar H(p_A) ) \dx, \\
&=  \int_{A\setminus B} \beta + \bar H(p_A) \dx  - \int_{B\setminus A} \beta + \bar H(p_A)  \dx \\
&\le  \int_{A} \left(\beta + \bar H(p_A)\right)^+ \dx  - \int_{\Omega\setminus A} \left(\beta + \bar H(p_A)\right)^-  \dx
= \delta_A.
\end{aligned}
\end{multline*}
If $B$ is a solution of \eqref{eq004} then the claim follows. Otherwise, we take the supremum of $- J(B,u_B)$ on the left-hand side.
\end{proof}

\section{The non-strongly convex case} \label{sec_nonstronglyconvex}

Let us briefly comment on the non-strongly convex case. That is, we no longer assume the strong convexity of $g$ as in \ref{ass1_strongc}.
We will replace \ref{ass1_strongc} and \ref{ass1_smooth} by the following two assumptions.

\begin{enumerate}[label=\bfseries (A\arabic*')]
\setcounter{enumi}{3}
 \item\label{ass2_g}
 $\dom g$ is a bounded subset of $\R$,
 \item\label{ass2_Sstar}
	There is $q > 3$ such that $S^* \in \L(Y, L^q(\Omega))$, where $S^*\in \L(Y,L^2(\Omega))$ denotes the Hilbert space-adjoint of $S$.
\end{enumerate}

\ref{ass2_g} implies the solvability of \eqref{eq_p_a}. In addition, solutions $u_A$ of \eqref{eq_p_a} will be in $L^\infty(\Omega)$.
Due to the missing strong convexity, we have to replace the assumption on $S^*S$ in \ref{ass1_smooth} by an assumption on $S^*$.
The $L^\infty(\Omega)$-regularity of optimal controls will allow us to work with a smaller exponent $q$ in \ref{ass2_Sstar} when compared to  \ref{ass1_smooth}.
Condition \ref{ass2_Sstar} is fulfilled for \cref{ex_elliptic,ex_parabolic}.

Note that we do not add assumptions that imply unique solvability of \eqref{eq_p_a}.

\begin{proposition}
 Let $A\subset \Omega$ be given.  Then there is a minimizer $u_A$ of \eqref{eq_p_a}.
 Moreover, $\chi_Au_A $ is also a minimizer of \eqref{eq_p_a}.
\end{proposition}
\begin{proof}
 Due to \ref{ass2_g} minimizing sequences of $u\mapsto J(A,u)$ are bounded in $L^\infty(\Omega)$.
 Then the proof of existence follows as in \cref{prop_existence}.
 The last claim is a consequence of \eqref{eq_J_chi}.
\end{proof}

In the sequel, we will assume that a solution $u_A$ of \eqref{eq_p_a} satisfies $\chi_Au_A=u_A$.
Due to the previous result, this is not restriction at all, as for every minimizer $u_A$ also $\chi_Au_A$ is a minimizer.
Let us start with a replacement of \cref{lem_bound_uy} and \cref{cor_linfty_bound}.

\begin{lemma}\label{lem_bounds_nsc}
 There is $M>0$ and $P'>0$  such that
 \[
  \|y_A - y_d\|_Y  \le M
 \]
 and
 \[
  \|u_A\|_{L^\infty(\Omega)}\le P' , \quad \|p_A\|_{L^q(\Omega)} \le P'
 \]
for all $A\subset \Omega$ and all solutions $(u_A,y_A,p_A)$ of \eqref{eq_p_a}.
Here, $q$ is as in \ref{ass2_Sstar}.
\end{lemma}
\begin{proof}
 The bound of $y_A$ can be obtained as in \cref{lem_bound_uy}, the  bounds of $u_A$ and $p_A$
 are consequences of \ref{ass2_g} and \ref{ass2_Sstar}.
\end{proof}

Due to the missing strong convexity of $g$, we cannot expect stability of controls as in \cref{thm_stability}.
Here, we have the following replacement.

\begin{theorem}\label{thm_stability_nsc}
Assume \ref{ass1_Omega}, \ref{ass1_y}, \ref{ass1_g}, \ref{ass2_g}, \ref{ass2_Sstar}.
Then there is a constant $K'>0$ such that for
all $A,B\subset \Omega$
\[
  \|p_A-p_B\|_{L^q(\Omega)} + \|y_B-y_A\|_Y\le K' \|\chi_A-\chi_B\|_{L^1(\Omega)}^{\frac12(1-\frac1q)},
\]
where $(u_A,y_A,p_A)$ and $(u_B,y_B,p_B)$ are solutions of \eqref{eq_p_a} and \PB, and $q$ is from  \ref{ass2_Sstar}.
\end{theorem}
\begin{proof}
From \ref{ass2_Sstar}, we get
$
\|p_A-p_B\|_{L^q(\Omega)} \le \|S^*\|_{\L(Y, L^q(\Omega))} \|y_A-y_B\|_{L^2(\Omega)}
$.
Define $\mu':=1 /  \|S^*\|_{\L(Y, L^q(\Omega))}^2$.
Let $q'$ be such that $\frac1{q'}+\frac1q=1$.
From the inequality of \cref{lem_diff_uy}, we obtain with H\"older's inequality
\[\begin{aligned}
\frac{\mu'}2 \|p_A-p_B\|_{L^q(\Omega)}^2 + \frac12\|y_B-y_A\|_Y^2
& \le \int_\Omega  (\chi_A-\chi_B)(p_Au_B - p_Bu_A)\dx\\
& \le 2(P')^2 \|\chi_A-\chi_B\|_{L^1(\Omega)}^{1-\frac1q} ,
\end{aligned}
\]
where $P'$ is from \cref{lem_bounds_nsc},
and the claim is proven.
\end{proof}

This stability result has to replace \cref{thm_stability} in the proof of \cref{thm_exp_valuef}.
The result corresponding to the latter theorem now reads as follows.
Note that due to \ref{ass2_g}, $\bar H(p)$ is well-defined and finite for all $p\in \R$.

\begin{theorem}\label{thm_exp_valuef_nsc}
Assume \ref{ass1_Omega}, \ref{ass1_y}, \ref{ass1_g}, \ref{ass2_g}, \ref{ass2_Sstar}, \ref{ass1_beta}.
Let $A,B\subset \Omega$, and let $(u_A,y_A,p_A)$ and $(u_B,y_B,p_B)$ be solutions of \eqref{eq_p_a}
and \PB.
Then it holds
\begin{multline*}
\left|J(A,u_A) - J(B,u_B) - \int_\Omega (\chi_A - \chi_B) (\beta + \bar H(p_B)) \dx \right| 
+  \frac12\|y_B-y_A\|_Y^2
\le P'K'\|\chi_A-\chi_B\|_{L^1(\Omega)}^{\frac32(1-\frac1q)},
\end{multline*}
where $P'$, $K'$, $q$ are from \cref{lem_bounds_nsc} and \ref{ass2_Sstar}, respectively.
\end{theorem}
\begin{proof}
 We can proceed exactly as in the proof of \cref{thm_exp_valuef} but now with $\mu=0$. Only the estimate \eqref{eq3701} has to be modified.
The estimate of \cref{lem_barH_lipschitz} has to be changed to
\begin{equation}\label{eq_est_barH_nsc}
  \| \bar H(p_A) -  \bar H(p_B)\|_{L^q(\Omega)} \le P'\|p_A-p_B\|_{L^q(\Omega)}
\end{equation}
using the $L^\infty(\Omega)$-bound of optimal controls in \cref{lem_bounds_nsc}, as well as the estimate of $\bar H$ from the proof of \cref{lem_barH_lipschitz}.
Then the error term of \eqref{eq3701} can be estimated using \cref{eq_est_barH_nsc} and \cref{thm_stability_nsc} as
\[
 \int_\Omega (\chi_A - \chi_B) (\bar H(p_A)- \bar H(p_B)) \dx
  \le \|\chi_A-\chi_B\|_{L^1(\Omega)}^{1-\frac1q}  \cdot P' \|p_A-p_B\|_{L^q(\Omega)} \le  P'K' \|\chi_A-\chi_B\|_{L^1(\Omega)}^{\frac32(1-\frac1q)}.
 \]
The claimed estimate can now be obtained with the same arguments as in the proof of \cref{thm_exp_valuef}.
\end{proof}

\begin{theorem}\label{thm_top_der_nsc}
Assume \ref{ass1_Omega}, \ref{ass1_y}, \ref{ass1_g}, \ref{ass2_g}, \ref{ass2_Sstar}, \ref{ass1_beta}.
Let $B\subset \Omega$. 

Then for almost all $x \in \Omega$ the topological derivative $ DJ(B)(x)$ exists, and it is given by
the expression in \cref{thm_top_der}.
\end{theorem}
%
%

\section{Optimization method based on the topological derivative}
\label{sec_alg}

In this section, we introduce an optimization algorithm that is motivated by the work on the topological derivative.
Here, we work under the set of assumptions of \cref{thm_top_der} or \cref{thm_top_der_nsc}.

Let $A_k\subset \Omega$ be a given iterate together with solutions of \PP{A_k}. 
Let us define the residual in the optimality condition of \cref{thm_opt_con} as
\begin{equation}\label{eq_def_rhok}
 \rho_k := (DJ(A_k))^- .
\end{equation}
Let us denote the support of $\rho_k$ by
\[
 R_k :=  \{x: \ \rho_k(x) \ne 0\} .
\]
New iterates $A_{k+1}$ will now be defined by adding/removing points to/from $A_k$, where  $\rho_k$ is non-zero.
That is, we will choose $D_{k,t} \subset R_k$ and denote the candidate for a new iterate by $A_k \triangle D_{k,t}$.
Given $t\in (0,1]$, we select $D_{k,t}$ with the properties:
\begin{equation}\label{eq_def_dkt}
\begin{aligned}
 D_{k,t} &\subset R_k:  &
 \|\rho_k\|_{L^1(D_{k,t})} &\ge t  \|\rho_k\|_{L^1(\Omega)} ,\\
 && |D_{k,t}| &\le  t |R_k|.
\end{aligned}
\end{equation}
That is, we are looking for a set $D_{k,t}$ with prescribed bound on its measure that captures a certain part of the mass of the residual $\rho_k$.
Sets satisfying the conditions of \eqref{eq_def_dkt} exist, and can be found by solving
\begin{equation}\label{eq_def_dkt_optim}
 \max_{D\subset R_k: \ |D| \le tR_k} \| \rho_k \|_{L^1(D)}.
\end{equation}

In \cite{CeaGioanMichel1973}, such a problem is used to generate search directions.
A related procedure to compute such sets is given in \cite[Procedure 1, Lemma 9]{HahnLeyfferSager2023}.
Note that we need both conditions on $D_{k,t}$: the condition on $\|\rho_k\|_{L^1(D_{k,t})}$
will give descent of values of the functional $J$, while the condition on $|D_{k,t}|$ will help to control
the error in the expansion of the functional $J$.

\begin{remark}
\label{rem_algorithms}
Let us comment on related algorithms based on topological derivatives.

In the
seminal work \cite{CeaGioanMichel1974} a similar idea was developed.
In our notation, their algorithm reads: find $t>0$ and $D_{k,t} \subset R_k$
such that $ |D_{k,t}| \le t$ and $\|\rho_k \|_{L^1(D)} \ge \frac{ \tilde M }2t^2$,
where $\tilde M$ is larger than the Lipschitz constant of the Fr\'echet derivative of $J$, when considered as a function from $L^1(\Omega)$ to $\R$.
This choice of $D_{k,t}$ guarantees a sufficient decrease of $J(A_k)$.
However, the knowledge of this Lipschitz constant is necessary to implement this condition.
In addition, this Lipschitz condition is not satisfied for our problem, it would
imply that the remainder term in \cref{thm_exp_valuef} is of order $\|\chi_A-\chi_B\|_{L^1(\Omega)}^2$,
while our analysis works under a weaker estimate of the remainder.

Another popular algorithm is the choice
$D_{k,t} = \{x: \ |\rho_k(x)| \ge t\}$, where $t\in (0, \|\rho_k\|_{L^\infty(\Omega)})$,
see, e.g.,  \cite{CeaGarreauGuillaumeMasmoudi2000,GarreauGuillaumeMasmoudi2001,HassineJanMasmoudi2007,HintermullerLaurain2010},
where $t$ is chosen such that the new iterate respects a volume constraint or a descent condition.
This approach is successfully used in practice.
From the viewpoint of optimization, it has the following theoretical drawback: if $\rho_k$ is a constant function, then $D_{k,t}$ is either empty or equal to $R_k$,
and a linesearch in $t$ is not guaranteed to succeed.

An algorithm based on trust-region ideas can be found in \cite{HahnLeyfferSager2023,MannsHahnKirchesLeyfferSager2023}.
In these works binary control problems are considered. The method proposed there can also be
applied to our problema, and would lead to similar convergence results.

In \cite{AmstutzA2006} a simplified level-set method was introduced, where the level-set function
is updated using the topological derivative.
The corresponding linesearch method was analyzed in \cite{Amstutz2011}.
There, small values of the linesearch parameter may only lead to boundary variations.
\end{remark}

Let us first prove that \eqref{eq_def_dkt_optim} can be used to find sets satisfying the condition \eqref{eq_def_dkt} for fixed $t\in(0,1]$.

\begin{lemma}\label{lem_exist_sigmak}
Let $t\in (0,1]$.
There exist  sets $D_{k,t}$ satisfying conditions \eqref{eq_def_dkt}.
Problem \eqref{eq_def_dkt_optim} is solvable, and every solution $D_{k,t}$ of
 \eqref{eq_def_dkt_optim} satisfies \eqref{eq_def_dkt}.
\end{lemma}
\begin{proof}
The existence of sets $D_{k,t}$ satisfying conditions \eqref{eq_def_dkt} with equality is a consequence of the Lyapunov convexity theorem for vector measures \cite[Corollary IX.5]{DiestelUhl1977}.
Consequently, solutions of \eqref{eq_def_dkt_optim} satisfy \eqref{eq_def_dkt}.
Let us proof the solvability of \eqref{eq_def_dkt_optim}.
 Given $s\ge0$ define
 \[
  B_{>s}:=\{ x: \ |\rho_k(x)| > s \}, \ B_{\ge s}:=  \{ x: \  |\rho_k(x)| \ge s \} .
 \]
 Then the monotonically decreasing functions $s\mapsto |B_{>s}|$ and $s\mapsto |B_{\ge s}|$ are  continuous from the right and from the left, respectively.
 In addition, we have $\lim_{s\to+\infty} |B_{>s}| = \lim_{s\to+\infty} |B_{\ge s}|=0$.

 Assume that there is $s\ge0$ and a set $B$ with $ B_{>s} \subset B \subset B_{\ge s}$ and $|B| = tR_k$.
 We will now argue that $B$ is a solution of \eqref{eq_def_dkt_optim}.
 Let us take $D\subset R_k$ with $|D| \le tR_k$.
 Then we get using the definitions of $B$, $B_{>s}$, $B_{\ge s}$
 \[\begin{aligned}
   \int_D |\rho_k| \dx - \int_B |\rho_k| \dx &= \int_{D \setminus B} |\rho_k| \dx - \int_{B \setminus D} |\rho_k| \dx  \\
   &\le s |D \setminus B| - s | B \setminus D| = s ( |D| -|D \cap B|  -| B ) \\
   &= s( |D| - tR_k) - s|D \cap B|\le0,
\end{aligned}\]
and $B_{>s}$ solves \eqref{eq_def_dkt_optim}.

Hence, if there is $s\ge0$ such that $|B_{>s}|=tR_k$ or $|B_{\ge s}|=tR_k$, then $B_{>s}$ or $B_{\ge s}$ is a solution of \eqref{eq_def_dkt_optim}.
If this is not the case, then there is $s>0$ such that $|B_{>s}|< tR_k <|B_{\ge s}|$.
Since the measure space is atom-free, we can choose $C\subset B_{\ge s} \setminus B_{>s}$ with $|C|=tR_k-|B_{>s}|$ by the Sierpinski theorem \cite[Corollary 1.12.10]{Bogachev2007}.
And $B_{>s} \cup C$ is a solution of \eqref{eq_def_dkt_optim} as argued above.
\end{proof}

The symmetric difference of $A_k$ and $D_{k,t}$ is a candidate for the new iterate.
Let us set
\begin{equation}\label{eq_def_akt_symm}
 A_{k,t} := A_k \triangle D_{k,t}.
\end{equation}
%
%
The step-size $t$ in \eqref{eq_def_dkt_optim} will be determined by an Armijo-like line-search
incorporating the  descent condition
\begin{equation} \label{eq_ls_descent}
 J(A_{k,t}) \le J(A_k) + \sigma \int_\Omega |\chi_{A_{k,t}}-\chi_{A_k}| \cdot \rho_k \dx
\end{equation}
where $\sigma\in (0,1)$ is a fixed parameter.

The validity of this approach is enabled by
the following observation, which shows that \eqref{eq_ls_descent} ensures descent of $J$.
\begin{lemma}\label{lem_diffchi_rho}
Let $D_{k,t}$ satisfy \eqref{eq_def_dkt}, let $A_{k,t}$ as in \eqref{eq_def_akt_symm}.
Then it holds
 \[
  \int_\Omega  |\chi_{A_{k,t}}-\chi_{A_k}| \cdot \rho_k \dx  = - \|\rho_k \|_{L^1(D_{k,t})} \le - t \|\rho_k\|_{L^1(\Omega)}.
 \]
\end{lemma}
\begin{proof}
This follows directly from the properties of $\rho_k$ and $D_{k,t}$, see \eqref{eq_def_rhok}.
%
\end{proof}

Let us prove that there are step-sizes $t$ such that the descent condition is satisfied.
Here, the following estimate of $J(A_{k,t}) - J(A_k)$ will prove useful, which is a consequence of the results of the previous sections.

\begin{lemma}\label{lem_diff_J_reduced}
There is $C>0$ and $\nu>0$ such that for all $A_k \subset \Omega$, $t\in (0,1]$, $A_{k,t}$ as in \eqref{eq_def_akt_symm}, it holds
 \[
   J(A_{k,t}) - J(A_k) \le \int_\Omega |\chi_{A_{k,t}}-\chi_{A_k}| \cdot \rho_k \dx +  C \|\chi_{A_{k,t}} - \chi_{A_k}\|_{L^1(\Omega)}^{1+\nu}.
 \]
\end{lemma}
\begin{proof}
As a consequence of the expansion of $J$ in  \cref{thm_exp_valuef,thm_exp_valuef_nsc} and of the definition of $DJ$ in \cref{thm_top_der}, we get
\[
 J(A_{k,t}) - J(A_k) \le \int_\Omega (\chi_{A_{k,t}}-\chi_{A_k})\sigma(A_k) DJ(A_k)\dx +  C \|\chi_{A_{k,t}} - \chi_{A_k}\|_{L^1(\Omega)}^{1+\nu},
\]
where $\sigma(A_k)$ is as in \cref{thm_top_der}.
Note that $|\chi_{A_{k,t}}-\chi_{A_k}| = (\chi_{A_{k,t}}-\chi_{A_k}) \sigma(A_k)$, which proves the claim.
\end{proof}


Now, we are in the position to prove the existence of step-sizes such that the descent condition \eqref{eq_ls_descent} is satisfied.

\begin{lemma}\label{lem_exist_tk}
Let $\sigma\in (0,1)$.
Assume $\rho_k \ne 0$.
Then there is $\tilde t_k >0$ such that for all $t\in (0,\tilde t_k]$ and all sets $D_{k,t}$ satisfying  \eqref{eq_def_dkt}
the set $A_{k,t}:=A_k \triangle D_{k,t}$ satisfies the descent condition \eqref{eq_ls_descent}.
\end{lemma}
\begin{proof}
Due to \eqref{eq_def_dkt}, we have
\[
\|\chi_{A_{k,t}} - \chi_{A_k}\|_{L^1(\Omega)} = |A_{k,t} \triangle A_k| = |D_{k,t}| \le t |R_k|.
\]
%
With the results of \cref{lem_diffchi_rho,lem_diff_J_reduced}, we get
 \begin{multline} \label{eq_lem64}
  J(A_{k,t}) - J(A_k) - \sigma \int_\Omega |\chi_{A_{k,t}}-\chi_{A_k}| \cdot \rho_k \dx\\
  \begin{aligned}
  &\le  (1-\sigma) \int_\Omega |\chi_{A_{k,t}}-\chi_{A_k}| \cdot \rho_k \dx +  C \|\chi_{A_{k,t}} - \chi_{A_k}\|_{L^1(\Omega)}^{1+\nu}\\
  & \le -(1-\sigma)  \|\rho_k\|_{L^1(\Omega)} t + C |R_k|^{1+\nu} t^{1+\nu}.
  \end{aligned}
\end{multline}
Clearly, the right-hand side
is negative for $t$ small enough.
\end{proof}

\begin{algorithm}[htbp]
\begin{algorithmic}
 \State Choose $\tau\in (0,1)$, $\sigma\in(0,1)$, $A_0 \subset \Omega$, $\dtol\ge0$. Set $k:=0$.

 \Loop \Comment{Gradient descent}

 \State Compute a solution $(u_k,y_k,p_k)$ of \PP{A_k}.
 \State Compute $\rho_k$ as in \eqref{eq_def_rhok}.
 \If{$\|\rho_k\|_{L^1(\Omega)} \le \dtol$}  \Comment{Termination criterion}
	\State  \textbf{return} $A_k$
 \EndIf

 \State $t:=1$.

 \Loop \Comment{Armijo line-search}
 \State Compute $D_{k,t}$ as a solution of \eqref{eq_def_dkt_optim}.
 \State Compute $J(A_{k,t})$ for $A_{k,t} = A_k \triangle D_{k,t}$.
 \If{$A_{k,t}$ satisfies \eqref{eq_ls_descent} } \State \textbf{break}\EndIf
 \State $t:= \tau \cdot t$.

 \EndLoop
 \State $t_k := t$.  \Comment{Update}
 \State $A_{k+1} := A_k \triangle D_{k,t}$.
\State $k:=k+1$.
\EndLoop
\end{algorithmic}
\caption{Topological gradient descent algorithm}
\label{alg_top}
\end{algorithm}

The resulting algorithm is \cref{alg_top}. Let us comment on it in detail.
The algorithm stops if $\|\rho_k\|_{L^1(\Omega)} \le \dtol$ for some prescribed tolerance $\dtol\ge0$.
This is motivated by \cref{lem_error_est}: if $\|\rho_k\|_{L^1(\Omega)}$ is less than some tolerance $\dtol>0$ then $A_k$ is
a $\dtol$-solution of \eqref{eq004}.
Due to \cref{lem_exist_tk}, the Armijo line-search will terminate in finitely many steps, and the algorithm is well-defined.
If the algorithm does not stop after finitely many iterations, then it will produce an infinite sequence of sets $(A_k)$, such that $(J(A_k))$ is monotonically decreasing.
We have the following basic convergence result.

\begin{lemma}\label{lem_conv_basic}
Let $(A_k)$ be an infinite sequence of iterates of \cref{alg_top}.
Then it holds
\[
  \sum_{k=0}^\infty t_k \|\rho_k\|_{L^1(\Omega)} < +\infty.
\]
\end{lemma}
\begin{proof}
Due to the descent condition \eqref{eq_ls_descent} and the result of \cref{lem_diffchi_rho}, we have the chain of inequalities
 \[
 J(A_{k+1}) - J(A_k) \le  \sigma \int_\Omega |\chi_{A_{k+1}} - \chi_{A_k}| \cdot \rho_k \dx
 \le - \sigma  t_k \|\rho_k\|_{L^1(\Omega)} < 0.
 \]
Since $J$ is bounded from below, we can sum the above inequalities for $k=0, \dots$, which proves the claim.
\end{proof}

\begin{theorem}\label{thm_convergence}
 Let $(A_k)$ be the iterates of \cref{alg_top}. Then exactly one of the following statements is true:
 \begin{enumerate}
  \item \label{exit1}  The algorithm returns after $m\in \N$ iterations with a $\dtol$-solution $A_m$ of  \eqref{eq004}.
  \item \label{exit2} The algorithm produces  an infinite sequence $(A_k)$
  with $\lim_{k\to\infty} \|\rho_k\|_{L^1(\Omega)} = 0$,
 and $(A_k)$ is a minimizing sequence of \eqref{eq004}.
\end{enumerate}
In particular, if $\dtol>0$ then the algorithm terminates after finitely many iterations.
\end{theorem}
\begin{proof}
 Suppose the algorithm returns after $m$ iterations. Then $\|\rho_m\|_{L^1(\Omega)} \le \dtol$, and by \cref{lem_error_est} the set $A_m$ is a $\dtol$-solution $A_m$ of  \eqref{eq004}.

 Now suppose that the algorithm produces  an infinite sequence $(A_k)$  of iterates.
 Let $k$ be such that $t_k < 1$. Due to the line-search procedure of \cref{alg_top},
 it follows that $t:=t_k / \tau\le 1$ violates the descent condition \eqref{eq_ls_descent}, that is
 \[
 0\le   J(A_{k,t}) - J(A_k) - \sigma \int_\Omega |\chi_{A_{k,t}}-\chi_{A_k}| \cdot \rho_k \dx.
 \]
 Using estimate \eqref{eq_lem64}, this implies
 \[
 0 \le -(1-\sigma)  \|\rho_k\|_{L^1(\Omega)} t + C |R_k|^{1+\nu} t^{1+\nu},
 \]
 or equivalently
 \[
  (1-\sigma)  \|\rho_k\|_{L^1(\Omega)} \le C |R_k|^{1+\nu} t^\nu = C |R_k|^{1+\nu} \tau^{-\nu} \ t_k^\nu.
 \]
 Note that $|R_k|\le |\Omega|$.
Together with the result of \cref{lem_conv_basic}, we obtain
\[
 \sum_{k:\,t_k=1}  \|\rho_k\|_{L^1(\Omega)} + \sum_{k:\, t_k<1} \|\rho_k\|_{L^1(\Omega)} ^{1+\frac1\nu} <+\infty,
\]
which results in $\lim_{k\to\infty} \|\rho_k\|_{L^1(\Omega)} = 0$.
Hence, the algorithm stops after finitely many iterations if $\dtol>0$.
\end{proof}

The sequence $(\chi_{A_k})$ of characteristic functions of the iterates admits weak-star converging subsequences in $L^\infty(\Omega) = L^1(\Omega)^*$.
However, the corresponding limits are not guaranteed to be characteristic functions.
If a subsequence converges weak-star to a characteristic function, i.e., $\chi_{A_{k'}} \rightharpoonup^* \chi_A$ in $L^\infty(\Omega)$,
then the convergence is strong in every $L^q(\Omega)$, $q<\infty$, and the limit $A$ is a solution of \eqref{eq004}.
If \eqref{eq004} is unsolvable then weak-star sequential limit points of $(\chi_{A_k})$ cannot be characteristic functions.

A similar convergence result for gradient descent without linesearch can be found in \cite[Theorem 2.1]{CeaGioanMichel1973},
whereas convergence of a trust-region-type algorithm can be found in \cite[Theorem 4.1]{MannsHahnKirchesLeyfferSager2023}.

\begin{remark}
\label{rem_conv_assumptions}
\cref{alg_top} can be generalized to optimization problems of the type $\min_{A\subset \Omega} J(A)$,
where the minimum is taken over measurable subsets $A$.
In order to apply the above analysis, we would need the following ``differentiability'' condition on $J$:
there is $\eta: [0,\infty) \to [0,\infty)$ with $\lim_{t\searrow0} \frac{\eta(t)}t =0$ such that
for each $A\subset \Omega$ there is $DJ(A)\in L^1(\Omega)$ such that
\begin{equation}\label{eq_general_J}
 \left| J(B) - J(A) - \int_\Omega (\chi_B-\chi_A)DJ(A) \dx \right| \le \eta( |B \triangle A| )
 \quad \forall B\subset \Omega.
\end{equation}
Then the result of \cref{lem_exist_tk} is valid, where assumption \eqref{eq_general_J} would act as a substitute for \cref{lem_diff_J_reduced}.
The statement of \cref{thm_convergence} has to be changed to: either $\rho_m=0$ for some finite $m$ or $\lim_{k\to\infty} \|\rho_k\|_{L^1(\Omega)} = 0$.

Note that condition \eqref{eq_general_J} is weaker than the assumptions of \cite{CeaGioanMichel1973,CeaGioanMichel1974} and \cite[Assumption 1.1(a)--(c)]{MannsHahnKirchesLeyfferSager2023}.
In the latter reference that stronger assumption was used to prove a statement analogous to \cref{thm_convergence}.
\end{remark}

\section{Numerical experiments}

\subsection{Optimal control problem with $L^0$-control cost}

Let us report about numerical results of the application of \cref{alg_top}
to the following problem:
Minimize
\[
 \min  \frac12 \|y-y_d\|_{L^2(\Omega)}^2 + \frac\alpha2 \|u\|_{L^2(\Omega)}^2 + \beta \|u\|_0
\]
over all $(y,u) \in H^1_0(\Omega) \times L^2(\Omega)$ satisfying
\[
 -\Delta  y = u \quad \text{ a.e.\@ on }\Omega
\]
and
\[
  u_a \le u \le u_b\quad \text{ a.e.\@ on }\Omega.
\]
This corresponds to the abstract setting with the choices $S:=(-\Delta)^{-1} : L^2(\Omega) \to H^1_0(\Omega) \hookrightarrow L^2(\Omega)$, $Y:=L^2(\Omega)$, $g(u):= \frac\alpha2 u^2 + I_{[u_a,u_b]}(u)$,
$\beta(x):=\beta$.
Here, $I_C$ denotes the indicator function of the convex set $C$, defined by $I_C(x)=0$ for $x\in C$ and $I_C(x)=+\infty$ for $x\not\in C$.
The assumptions are all satisfied. In particular $g$ is strongly convex with modulus $\mu:=\alpha$.
In the numerical experiment, we used $\Omega \subset \R^2$ bounded, so that
assumptions \ref{ass1_smooth} and \ref{ass2_Sstar} are satisfied with $q=\infty$ due to Stampacchia's result \cite{Stampacchia1965}.

\subsubsection{Example from \cite{Wachsmuth2019}}
\label{sec_num11}

We choose $\Omega=(0,1)^2$. We used a standard finite-element discretization on a shape-regular mesh on $\Omega$.
State and adjoint variables (i.e., $y$, $p$) were discretized using continuous piecewise linear functions, while the
control variable was discretized using piecewise constant functions.
Let us remark that for the finest discretization, the control functions have $2,000,000$ degrees of freedom.
The subproblems \eqref{eq_p_a} were solved by a semismooth Newton implementation.
The parameters in the line-search of \cref{alg_top} were chosen to be $\tau=0.5$ and $\sigma=0.1$.
The algorithm was stopped if one of the following conditions was fulfilled: $\|\rho_k\|_{L^\infty(\Omega)}\le 10^{-12}$, the support of $\rho_k$
contained $\le 3$ elements, or the line-search failed to find a valid step-size.
Termination due to the latter condition can happen if the relevant quantities in \eqref{eq_ls_descent},
are very small so that errors in the inexact solve of the sub-problem \eqref{eq_p_a}
are of the same order.

In addition, we used the following data
\[
 y_d(x_1,x_2) = 10 x_1 \sin(5x_1) \cos(7x_2), \ \alpha = 0.01,\  \beta=0.01,\ u_a = -4, \ u_b=+4,
\]
which was also used in \cite{ItoKunisch2014,Wachsmuth2019}.
The computed optimal control, which  is obtained by the last iterate of \cref{alg_top} on the finest mesh,
can be seen in \cref{fig1}. Due to the presence of the $L^0$-term in the objective, the control is zero on a relatively large
part of $\Omega$.

\begin{figure}[htb]
\begin{center}
 \includegraphics[width=0.45\textwidth]{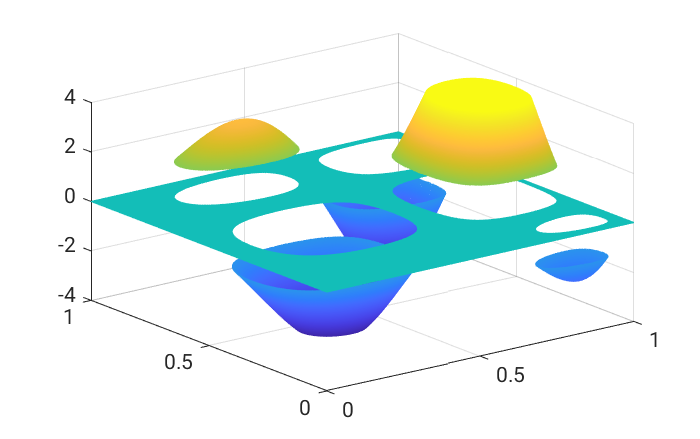}
\end{center}
\caption{Solution $u$ for $h=1.41\cdot 10^{-3}$, \cref{sec_num11} }
\label{fig1}
\end{figure}

The results of the computations for different meshes can be seen in \cref{table1}.
There, $h$ denotes the mesh-size of the triangulation, $J$ denotes the value of the functional $J$ at the final iterate,
similarly $\|\chi\|_{L^1(\Omega)}$ is the size of the support of the optimal control, and $\|\rho\|_{L^1(\Omega)}$ is
 the error estimate from the topological derivative at the final iteration.
The values corresponding to the mesh-size $h=2.83\cdot 10^{-3}$ are in agreement with those from \cite{Wachsmuth2019}.
For this example, all computations stopped due to the support of $\rho_k$ containing less than 3 elements.
In addition, for this example, the step-size $t=1$ was always accepted.
\cref{alg_top} was started with the initial choice $A_0=\Omega$.
As can be seen from \cref{table1}, the optimal values of $J$ and $\|\chi\|_{L^1(\Omega)} $ converge for $h\searrow0$, and $\|\rho\|_{L^1(\Omega)}\to0$ for $h\searrow0$.
According to \cref{thm_convergence}, this strongly suggests that the iterates are a minimizing sequence of \eqref{eq004}.
In the last column of \cref{table1}, we report about the number of iterations until the termination criterion is satisfied.
As can be seen, the numer of iterations rises midly after mesh refinement.


\begin{table}[htb]
\sisetup{table-alignment-mode = format, table-number-alignment = left}
\begin{center}
\begin{tabular}{S[table-format = 1.2e2]S[table-format = 1.3]S[table-format=1.5]S[table-format=1.2e2]S[table-format=1] }
\toprule
$h$ & $J$ & $\|\chi_A\|_{L^1(\Omega)}$ & $\|\rho\|_{L^1(\Omega)}$ & It\\
\midrule
 4.42e-02 & 4.712 & 0.43896 & 4.33e-03 & 2 \\
 2.21e-02 & 5.054 & 0.44299 & 2.12e-08 & 3 \\
 1.13e-02 & 5.216 & 0.44352 & 2.09e-08 & 3 \\
 5.66e-03 & 5.299 & 0.44432 & 2.04e-08 & 3 \\
 2.83e-03 & 5.340 & 0.44455 & 2.11e-11 & 4 \\
 1.41e-03 & 5.360 & 0.44460 & 4.05e-11 & 4 \\
\bottomrule
\end{tabular}
\caption{Results of optimization, \cref{sec_num11}}\label{table1}
\end{center}
\end{table}

Let us report about the influence of the choice of the initial guess $A_0\subset \Omega$.
Here we chose the following set of parameters: $y_d$ was as above, and
\[
 \alpha = 0.001,\  \beta=0.1,\ u_a = -40, \ u_b=+40.
\]
For this example, the method returned the same solution independent of the initial guess.
We depicted the iteration history for different choices of $A_0$ in \cref{fig4}.
In general, the method was faster when starting from $A_0=\Omega$ than from $A_0=\emptyset$.
As one can see from \cref{fig4}, the convergence of $\|\rho_k\|_{L^1(\Omega)}$ is stable with respect to mesh refinement.

\begin{figure}[htb]
\begin{center}
 \includegraphics[width=0.45\textwidth]{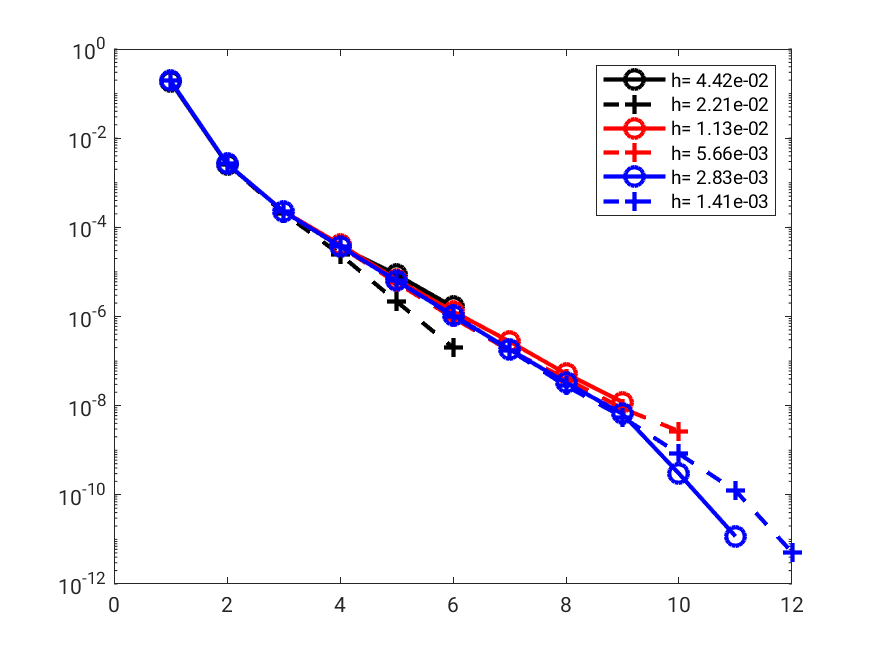}
 \includegraphics[width=0.45\textwidth]{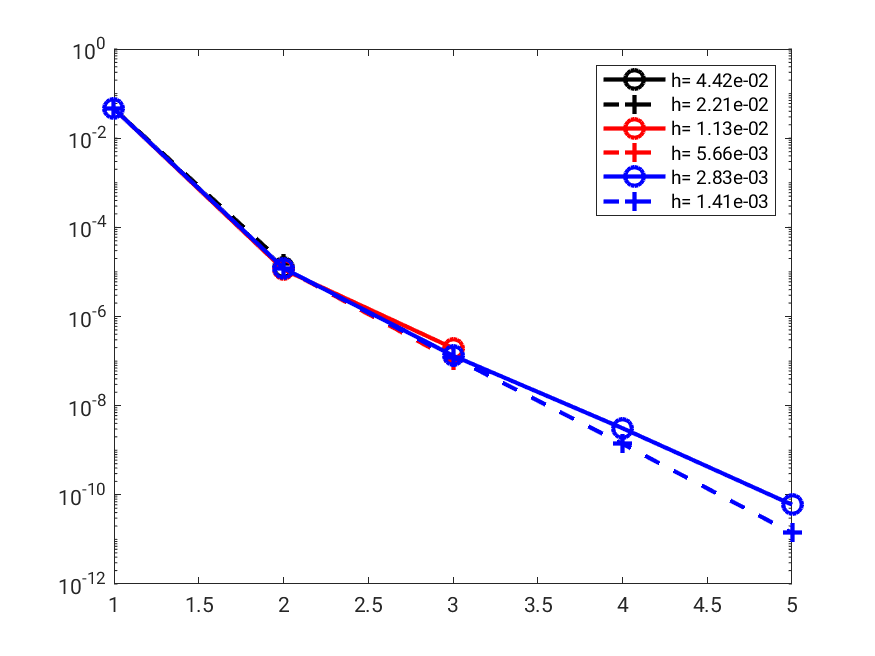}
\end{center}
\caption{Comparison of iteration history of $\|\rho_k\|_{L^1(\Omega)}$ for different choice of $A_0$: $A_0=\emptyset$ (left), $A_0=\Omega$ (right), \cref{sec_num11}}
\label{fig4}
\end{figure}

\subsubsection{An unsolvable problem}
\label{sec_num12}

Let us report about observations when applying our algorithm to an unsolvable problem.
This problem is taken from \cite[Section 4.5]{Wachsmuth2019}. It is very similar to the above problem. The partial differential equation is chosen in such
a way that a constant control $u$ leads to a constant solution $y$.
The problem reads as follows:
Minimize the functional
\[
  \frac12\|y-y_d\|_{L^2(\Omega)}^2 + \frac\alpha2 \|u\|_{L^2(\Omega)}^2 + \beta\|u\|_0,
\]
where $y$ denotes the weak solution of the elliptic partial differential equation
with Neumann boundary conditions
\[
 -\Delta y + y = u \quad \text{ in } \Omega, \quad \frac{\partial y}{\partial n}=0 \quad \text{ on } \partial \Omega.
\]
Let $\alpha>0$ and $\beta>0$ be given, and set
\[
 y_d(x) = -\sqrt{\frac \beta\alpha} - \sqrt{2 \alpha\beta}.
\]
As argued in \cite[Section 4.5]{Wachsmuth2019} this optimal control problem is unsolvable.
This implies that there is also no minimizer $A$ of the value function $J(A)$.

\begin{table}[htb]
\sisetup{table-alignment-mode = format, table-format = 2.2e2, table-number-alignment = left}
\begin{center}
\begin{tabular}{S[table-format = 1.2e1]S[table-format = 2.3]S[table-format=1.5]S[table-format=1.2e2]}
\toprule
$h$ & $J$ & $\|\chi_A\|_{L^1(\Omega)}$ & $\|\rho\|_{L^1(\Omega)}$ \\
\midrule
4.42e-2 & 10.014 & 0.70703 & 6.15e-9 \\
2.21e-2 & 10.014 & 0.70703 & 8.63e-10 \\
1.13e-2 & 10.014 & 0.70701 & 9.29e-10 \\
5.66e-3 & 10.014 & 0.70670 & 2.75e-9 \\
2.83e-3 & 10.014 & 0.70674 & 2.49e-9 \\
1.41e-3 & 10.014 & 0.70693 & 1.37e-9 \\
\bottomrule
\end{tabular}
\caption{Results of optimization, \cref{sec_num12}}\label{table2}
\end{center}
\end{table}

In the computations, we used
\[
 \beta = 0.01, \quad \alpha = 1000, \ \Omega = (0,1)^2.
\]
The results are shown in \cref{table2}.
There, the values of $J$, $\|\chi_A\|_{L^1(\Omega)}$, and $\|\rho\|_{L^1(\Omega)}$ are shown for the final iterate on
meshes with different mesh-sizes $h$.
In contrast to the results of \cite{Wachsmuth2019}, where resulting controls had always support equal to $\Omega$,
the measure of the sets $A_k$ was in the order of $0.7$ after a few steps of \cref{alg_top}.
In our experiments, the line-search took much more steps than in the previous example, and
only small modifications of the $A_k$ were accepted, resulting in very slow convergence.
While $\|\rho_k\|_{L^1(\Omega)}$ was very small after a few steps of the algorithm,
the support of $\rho_k$ never gets as small as for the previous example.
For this example, we stopped the algorithm after 100 iterations.
Still, according to  \cref{thm_convergence}, the algorithm
produces a minimizing sequence for the unsolvable example,
which is not the case for the thresholding method of  \cite{Wachsmuth2019}.

\subsection{Binary control problems}\label{sec_num2}

Following the ideas of \cite{Amstutz2010,Amstutz2011,HahnLeyfferSager2023,MannsHahnKirchesLeyfferSager2023}, we will apply our algorithm to a binary control problem,
where controls  can only take values in $\{0,+1\}$.
We will use a problem considered in \cite{Amstutz2010,Amstutz2011}, which reads:
Minimize
\[
 \min  \frac12 \|y-y_d\|_{L^2(\Omega)}^2 + \nu \|u\|_{L^1(\Omega)}
\]
over all $(y,u) \in H^1_0(\Omega) \times L^2(\Omega)$ satisfying
\[
 -\Delta  y = u \quad \text{ a.e.\@ on }\Omega
\]
and
\[
 u(x) \in \{0,1\} \text{ f.a.a.\@ } x\in \Omega.
\]
Hence, $u$ itself is a characteristic function. And the above problem can be written in our setting as:
Minimize
\[
 J(A,u) := \frac12 \|y-y_d\|_{L^2(\Omega)}^2 + \nu \int_A \dx
\]
over all $(y,u) \in H^1_0(\Omega) \times L^2(\Omega)$ satisfying
\[
 -\Delta  y = \chi_A u \quad \text{ a.e.\@ on }\Omega
\]
and the (trivial)
constraint
\[
 u = 1 \text{ a.e.\@ on }\Omega.
\]
This setting does not directly fit into our framework.
Still we can compute the topological derivative as follows.
The solution of $u\mapsto J(A,u)$ is given by $u_A\equiv1$, which greatly simplifies the computations of \cref{sec3}.
And we have the following result concering the topological derivative of the value function.

\begin{theorem}
The topological derivative $ DJ(B)(x)$
of the value function of the binary control problem
exists
 for almost all $x \in \Omega$, and is given by
 \[
  DJ(B)(x) = \sigma(B,x)(\beta(x) + p_B(x))
 \]
with $\sigma(B,x)$ as in \cref{thm_top_der}.
\end{theorem}
\begin{proof}
The result of \cref{lem_diff_J_final} in this situation has to be modified to
\[ 
J(A,u_A) - J(B,u_B)  +  \frac12\|y_B-y_A\|_Y^2 =\int_\Omega   (\chi_A-\chi_B)(\beta+ p_A )\dx.
\]
where we have used set $g=0$ and $\chi_Au_A-\chi_Bu_B = \chi_A-\chi_B$ in \eqref{eq3101}.
Since $p_A-p_B = S^*S(\chi_A-\chi_B)$, we have the estimate $\|p_A-p_B\|_{L^q(\Omega)} \le c \|\chi_A-\chi_B\|_{L^2(\Omega)} = c \|\chi_A-\chi_B\|_{L^1(\Omega)}^{1/2}$,
which replaces the result of \cref{thm_stability}.
Now the claim can be proven as in the proof of \cref{thm_top_der}.
\end{proof}

The topological derivative coincides with the result \cite[Corollary 3.2]{Amstutz2010}.
The computation of the topological derivative does not involve the solution of any optimization problem: given $A$, only $y_A$ and $p_A$ have to be computed.

Let us report about the results for the following choice of parameters, corresponding to Case 3 in \cite[Section 9]{Amstutz2010}:
\[
y_d = 0.05, \quad \nu = 0.002.
\]
The computed control on the finest discretization can be seen in \cref{fig5}, which agrees with  \cite[Figure 4]{Amstutz2010}.
The results of the optimization runs for different discretizations can be seen in \cref{table3}.
In all cases, the algorithm stopped due to a failed line-search. Nevertheless, the error quantity $\|\rho\|_{L^1(\Omega)}$ is very small,
and is decreasing with decreasing mesh-size. According to \cref{thm_convergence} this indicates that the algorithm produces a minimizing sequence.

\begin{table}[htb]
\sisetup{table-alignment-mode = format, table-format = 2.2e2, table-number-alignment = left}
\begin{center}
\begin{tabular}{S[table-format = 1.2e2]S[table-format = 1.3]S[table-format=1.5]S[table-format=1.2e2]S[table-format=2] }
\toprule
$h$ & $J$ & $\|\chi_A\|_{L^1(\Omega)}$ & $\|\rho\|_{L^1(\Omega)}$ & It\\
\midrule
 6.99e-02 & 1.799e-03 & 1.63770 & 5.60e-08 & 21 \\
 3.49e-02 & 1.872e-03 & 1.63794 & 4.16e-09 & 25 \\
 1.79e-02 & 1.909e-03 & 1.63808 & 3.93e-10 & 30 \\
 8.94e-03 & 1.928e-03 & 1.63806 & 5.36e-11 & 36 \\
 4.47e-03 & 1.938e-03 & 1.63802 & 1.51e-12 & 39 \\
 2.24e-03 & 1.943e-03 & 1.63802 & 3.28e-12 & 35 \\
\bottomrule
\end{tabular}
\caption{Results of optimization, \cref{sec_num2}} \label{table3}
\end{center}
\end{table}

\begin{figure}[htb]
\begin{center}
 \includegraphics[width=0.75\textwidth]{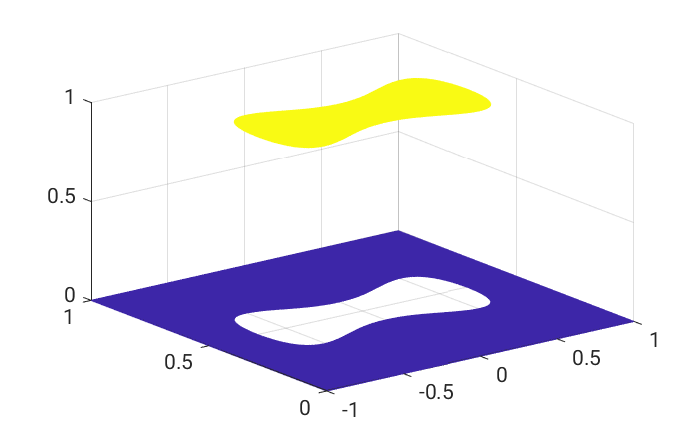}
\end{center}
\caption{Solution for $h=2.24\cdot 10^{-3}$, \cref{sec_num2} }
\label{fig5}
\end{figure}

We also implemented the level-set method from \cite{Amstutz2011, AmstutzA2006}.
All parameters were chosen as in \cite{Amstutz2011}.
We took the same example as above and performed the computations on the grid with $h=2.24\cdot 10^{-3}$.
The resulting iteration history of $ \|\rho_k\|_{L^1(\Omega)}$
for our method and the level-set method can be seen in \cref{fig6}.
In the level-set method, the domain $A$ is determined as $A:=\{x: \ \psi(x) \ge 0\}$,
where $\psi$ is the level-set function.
Here, it makes a difference, whether $\psi$ is discretized by piecewise constant (P$0$) or piecewise linear (P$1$)
finite elements. We implemented both variants.
Our method was implemented using piecewise constant functions for the control.
The computations in all three methods were stopped as soon as $ \|\rho_k\|_{L^1(\Omega)} < 10^{-11}$.
The results can be seen in \cref{fig6}: the black plus-signs correspond to the piecewise constant (P$0$) case, while
the blue x's refer to the piecewise linear (P$1$) case of the level-set method, where all integrals are computed exactly following the ideas of \cite{Hinze2005}.
While the level-set method in the piecewise linear case seems to be the fastest of these three methods,
the comparison to our method with piecewise constant discretization is a bit unfair, as the piecewise linear method
can resolve the interface much finer.
Still our method is faster than the piecewise constant variant of the level-set method.

\begin{figure}[htb]
\begin{center}
 \includegraphics[width=0.75\textwidth]{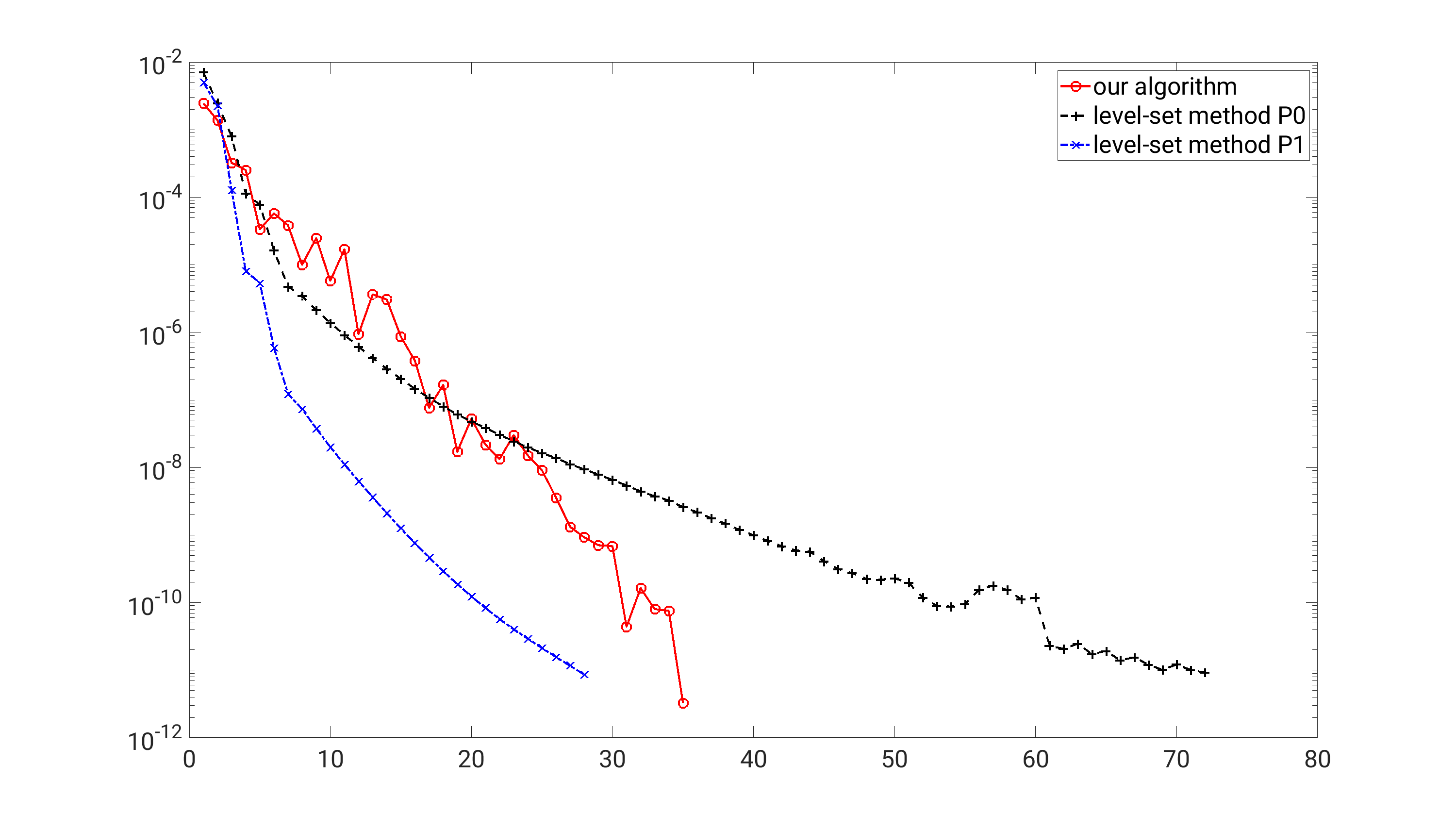}
\end{center}
\caption{Comparison of \cref{alg_top} (red circles) and level-set method \cite{Amstutz2011} (black +'s: P$0$, blue x's: P$1$): Iteration history of $ \|\rho_k\|_{L^1(\Omega)}$, \cref{sec_num2} }
\label{fig6}
\end{figure}

\section{Conclusion}

In this paper, we developed an algorithm to solve control problems with $L^0$-cost.
As showed in \cref{thm_convergence}, the algorithm produces a minimizing sequence,
which is remarkable as the $L^0$-optimization problem is not convex.
The algorithm was motivated by the concept of topological derivatives applied
to a suitable chosen sub-problem.

\section*{Acknowledgements}

The author thanks the anonymous referees, Bastian Dittrich, Fritz Schwob, Gerd Wachsmuth, and Daniel Walther  for remarks on an earlier versions of the manuscript.


\bibliographystyle{jnsao}
\bibliography{topological}

\end{document}